\documentclass[12pt]{article}

\usepackage{geometry} 
\usepackage{amssymb,latexsym,amsmath,amsthm,epstopdf,verbatim}
\usepackage{graphicx,epsfig,amssymb}
\usepackage[percent]{overpic}
\usepackage{color}

\newcommand{\R}{\mathbb{R}} 

\newcommand{\N}{{\mathbb N}}

\newtheorem{theorem}{Theorem}[section]
\newtheorem{lemma}[theorem]{Lemma}

\newtheorem{remark}[theorem]{Remark}

\begin{document}
\title{The validity of Whitham's approximation  for a Klein-Gordon-Boussinesq model}
\author{Wolf-Patrick D\"ull, Kourosh Sanei Kashani, Guido Schneider \\[5mm] 
Institut f\"ur Analysis, Dynamik und Modellierung, 
Universit\"at Stuttgart, \\ Pfaffenwaldring 57, 70569 Stuttgart, Germany}
\date{\today}

\maketitle

\begin{abstract}
In this paper
 we prove   the validity of  a long wave 
 Whitham approximation  for a system consisting of  a Boussinesq equation coupled with a Klein-Gordon equation. 
 The proof is based on an infinite series of normal form transforma\-tions and an energy estimate.
We expect that the concepts of this paper will be a part of a general approximation theory for Whitham's equations
which are especially used  in the description of 
slow modulations in time and space 
of periodic wave trains in general dispersive wave systems. 
\end{abstract}

\section{Introduction}

For dispersive wave systems there are various long wave approximations, the most prominent ones are KdV approximation and  the Whitham approximation.
For the Boussinesq equation 
\begin{equation} \label{bousseq1}
\partial_t^2u	= \partial_x^2u+\partial_t^2\partial_x^2u+\partial_x^2(u^2),
\end{equation}
with $ x,t,u(x,t)  \in \mathbb{R} $,
as a toy model, with the ansatz 
$$
u(x,t) = \varepsilon^2 A(\varepsilon(x-t) , \varepsilon^3 t) ,
$$
where $ 0 < \varepsilon \ll 1 $ is a small perturbation parameter and $ A(X,T) \in \mathbb{R} $,
we obtain the KdV equation
$$
\partial_T A = \partial_X^3 A/2 + \partial_X(A^2)/2,
$$
and with the ansatz 
$$
u(x,t) =  A(\varepsilon x , \varepsilon t) ,
$$
again with $ A(X,T) \in \mathbb{R} $,
we obtain the Whitham system 
$$
\partial_t^2u	= \partial_x^2u+ \partial_x^2(u^2)
$$
which can be written as a first order system  
$$
\partial_t u	= \partial_x v, \qquad 
\partial_t v 	= \partial_x u+ \partial_x (u^2)
$$
of conservation laws. The approximation of solutions of a dispersive system like the Boussinesq equation \eqref{bousseq1} by solutions to the KdV equation and the Whitham system, respectively, is called KdV approximation and Whitham approximation, respectively.

Both approximations describe the modes which are concentrated at the wave number $ k = 0 $ in Fourier space. See Figure \ref{figure1neu}.

\begin{figure}[htbp]
\begin{center}
       \begin{overpic}[width=250 pt]{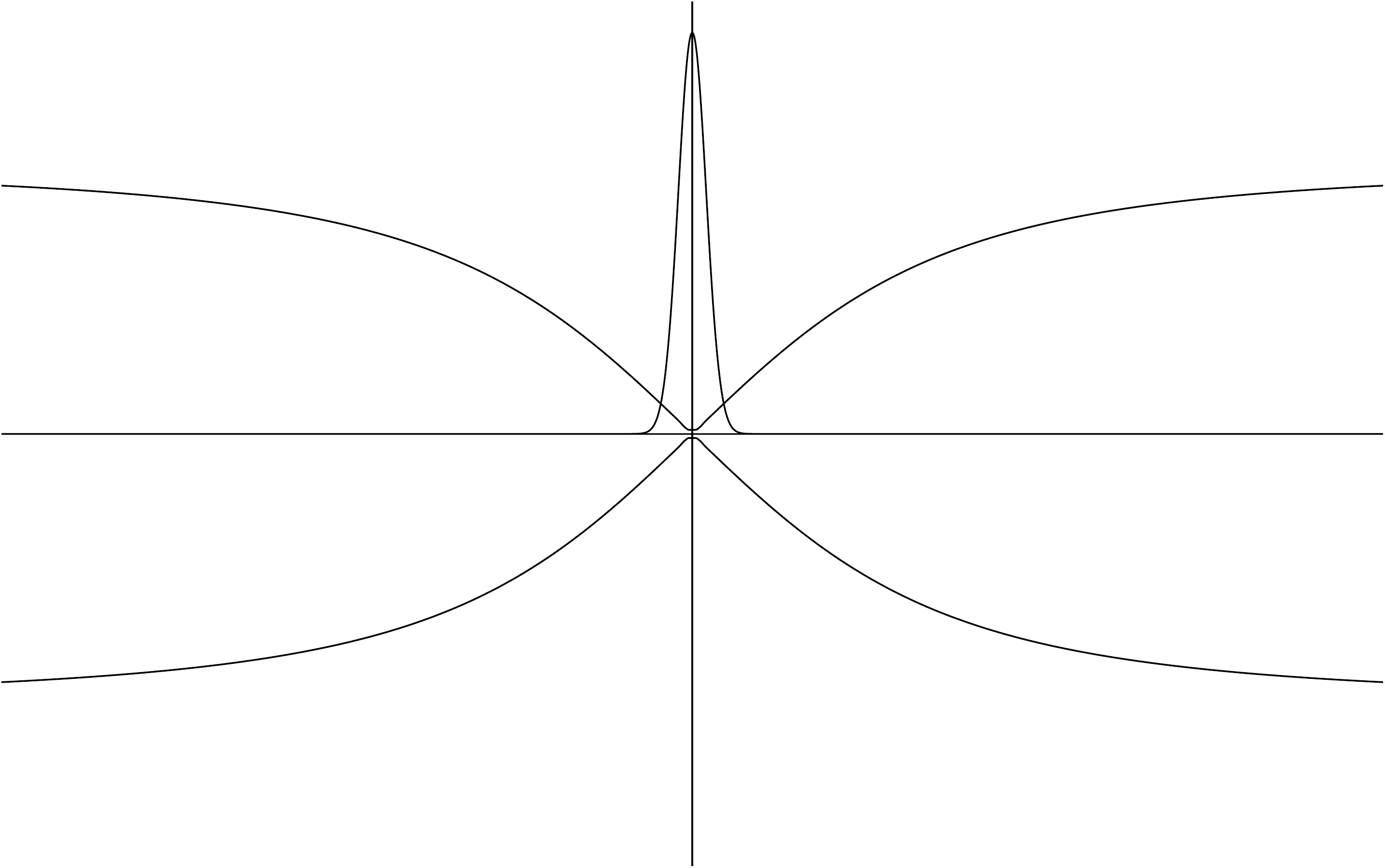}
       \put(95,25){$k$}
       \put(95,45){$\omega_1$}
       \put(95,15){$\omega_{-1}$}
       \put(55,55){$ \varepsilon \widehat{A}(\frac{k}{\varepsilon})  $}
       \end{overpic}
    \caption{\label{figure1neu} Curves of eigenvalues  $ \omega_{\pm 1} $ for the Boussinesq model \eqref{bousseq1}. Since the Fourier transform of  $ \varepsilon^2 A(\varepsilon x) $ is given 
    by $ \varepsilon \widehat{A}(\frac{k}{\varepsilon}) $ the modes  of the KdV and 
    Whitham approximation are strongly concentrated at the wave number $ k = 0 $.}
\end{center}
   \end{figure}
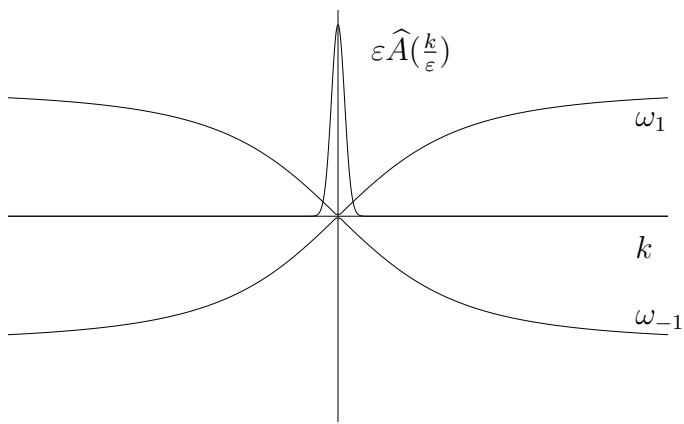

The error made by these approximations can easily be estimated in both 
cases by simple energy estimates, cf. \cite[\S 20]{SUbook2016}, similar to \cite{GS01} by using the long wave character of the approximation.
The fact that the nonlinear terms vanish at the wave number $ k = 0 $, too, allows 
to construct an energy which allows to control the $ \mathcal{O}(\varepsilon^2) $
terms on the  $ \mathcal{O}(1/\varepsilon^3) $ time scale, respectively, to control the $ \mathcal{O}(1) $
terms on the  $ \mathcal{O}(1/\varepsilon) $ time scale.
For the justification of the KdV approximation 
this approach has been used for the water wave problem in  \cite{Cr85,SW00,SW02,D12}
and for the FPU system in \cite{SW99equadiff}.

The  system changes dramatically if more than the two curves of eigenvalues, 
as drawn in Figure \ref{figure1.1}, are present in the problem.
Examples are the water wave problem over a periodic bottom topography,
the poly-atomic FPU model and slow modulations in time and space 
of periodic wave trains.
The simplest toy problem with this property is a system consisting of a 
Boussinesq equation coupled with a  
Klein-Gordon 
equation, namely
\begin{align} 
\partial_t^2u	&= \partial_x^2u+\partial_t^2\partial_x^2u+\partial_x^2(u^2+2uv+v^2),&\label{eq1}&\\
\partial_t^2v	&= \partial_x^2v-2 v-(u^2+2uv+v^2),\label{eq2}
\end{align}
with  $x,t,u(x,t), v(x,t) \in \mathbb{R}$. This set of equations is called 
 KGB system in the following.
The linearization of the KGB model around the trivial solution $u=v=0$ 
has plane wave solutions of the form
$u(x,t)=e^{ikx+i\omega_{\pm1}(k)t}$ and $v(x,t)=e^{ikx+i\omega_{\pm2}(k)t}$ with
$$
\omega_{\pm1}(k)= \pm \rm{sign}(k) \sqrt{\frac{k^2}{k^2+1}} \quad \text {and} \quad  \omega_{\pm2}(k)= \pm \sqrt{k^2+2}.
$$
The curves of eigenvalues are plotted in Figure \ref{figure1.1}.

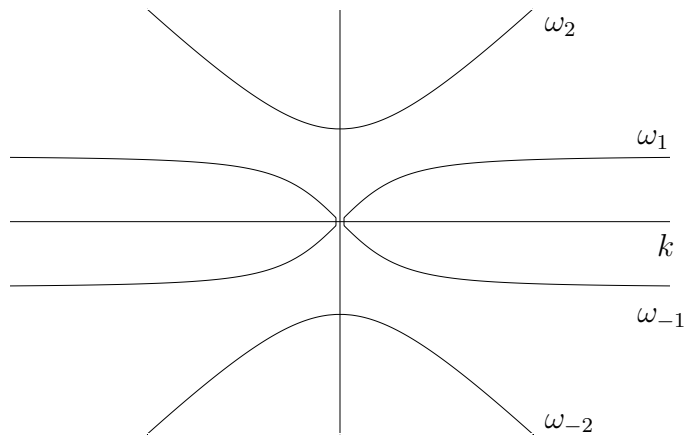
\begin{figure}[htbp]
\begin{center}
       \begin{overpic}[width=250 pt]{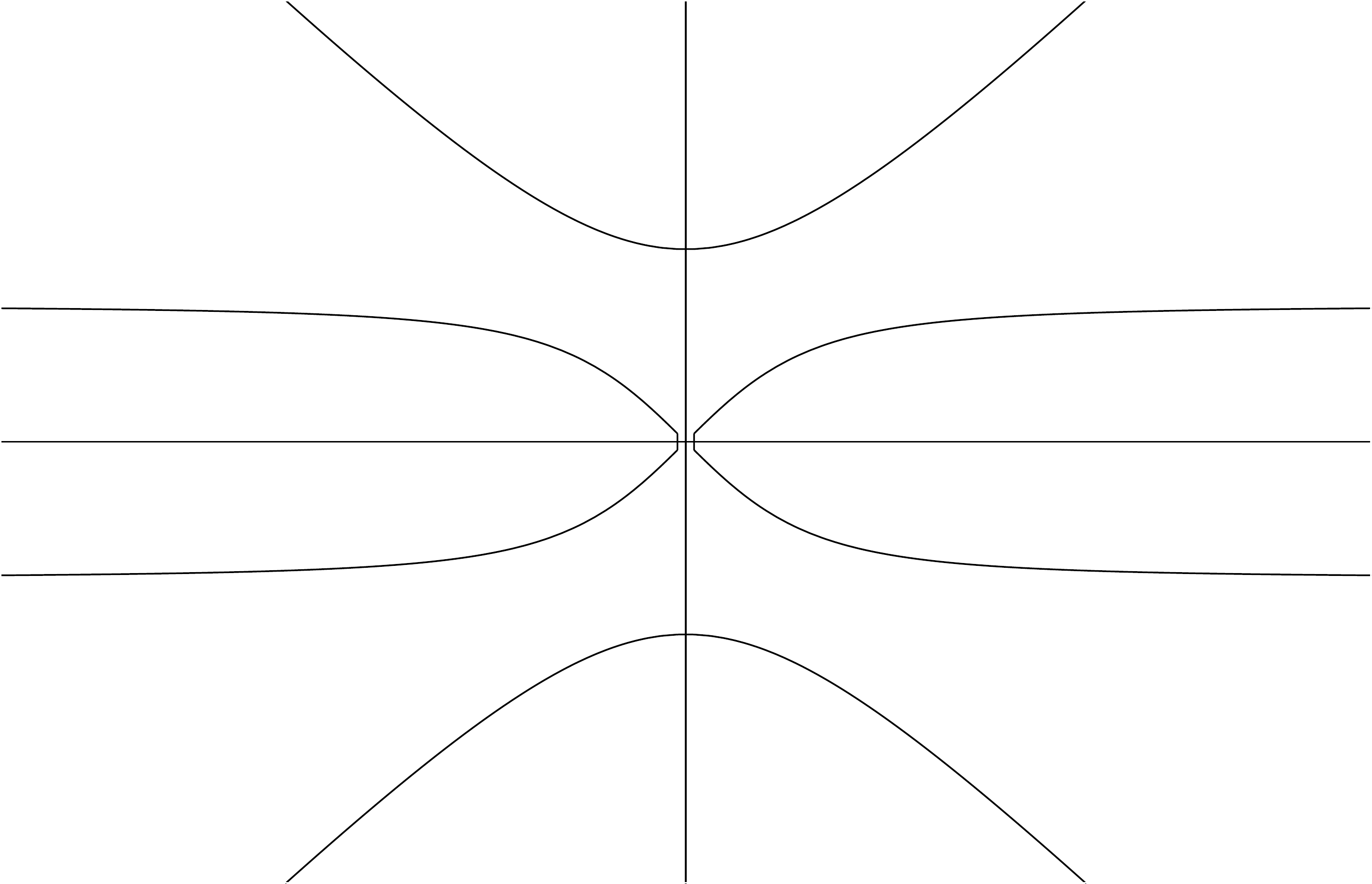}
       \put(98,27){$k$}
       \put(95,44){$\omega_1$}
       \put(95,17){$\omega_{-1}$}
       \put(81,61){$\omega_2$}
       \put(81,1){$\omega_{-2}$}
       \end{overpic}
    \caption{\label{figure1.1} Curves of eigenvalues $\omega_{\pm1}$ and $\omega_{\pm2}$ of the linearized KGB model}
\end{center}
   \end{figure}

In this situation simple energy estimates in general are no longer sufficient 
to justify the long wave approximations
 since now quadratic terms 
are present in the problem which no longer vanish at the wave number $ k= 0 $.
For a KdV approximation with a  combination of 
one normal form transformation and  energy estimates this problem 
has been solved. See \cite{ChSch11} for the KGB system and 
\cite{CCPS12,Wrightgroup} for the poly-atomic FPU model.

For the Whitham approximation a new serious difficulty occurs, namely the fact that 
due to the scaling of the ansatz for the derivation of the Whitham system
infinitely many 
normal form transformations  have to be performed
instead of only one in the KdV case. 
It is the goal of this paper to show that this approach really works and that 
an approximation theorem for the Whitham approximation of the KGB system
can be established.

The ansatz for the derivation of the Whitham system from the KGB system  has the form
\begin{equation} 
\label{eq4} 
\psi_u^{\text{Whitham}}(x,t)=U(\varepsilon x,\varepsilon t) \ \ \text{and} \ \ \psi_v^{\text{Whitham}}(x,t)=V(\varepsilon x,\varepsilon t) \,.
 \end{equation}
Inserting this ansatz into \eqref{eq1} and \eqref{eq2} we find for 
\begin{eqnarray*}
\text{Res}_u(u,v) & = & -\partial_t^2u+\partial_x^2u+\partial_t^2\partial_x^2u+\partial_x^2(u^2+2uv+v^2),
\\
Š\text{Res}_v(u,v)& =& -\partial_t^2v+\partial_x^2v-2v-(u^2+2uv+v^2) 
\end{eqnarray*}
that 
\begin{eqnarray*}
\text{Res}_u(\psi_u^{\text{Whitham}},\psi_v^{\text{Whitham}}) & = &  \varepsilon^2(-\partial_T^2U+\partial_X^2U+\partial_X^2(U^2+2UV+V^2))+\varepsilon^4\partial_T^2\partial_X^2U,
\\
\text{Res}_v(\psi_u^{\text{Whitham}},\psi_v^{\text{Whitham}})& =& -2V-(U^2+2UV+ V^2)+\varepsilon^2(-\partial_T^2V+\partial_X^2V).
\end{eqnarray*}
Hence equating the coefficients of $\varepsilon^0$  in $ \text{Res}_v $ to zero yields
\[
2V+U^2+2UV+V^2=0\,,
\]
and so $ V =H(U)=- U^2/2+\mathcal{O}(U^3)$ due to  the implicit function theorem
for $ U $ and $ V $ of $ \mathcal{O}(1)$,  but sufficiently small.
Equating the coefficients of $\varepsilon^2$ in $ \text{Res}_u $ to zero gives
\begin{equation} 
\label{eqn1} 
-\partial_T^2U+\partial_X^2U+\partial_X^2(U^2+2UV+V^2)=0.
 \end{equation}
By substituting $V=H(U)$ into \eqref{eqn1} we find 
\begin{equation}\label{eq5}
-\partial_T^2U+\partial_X^2U+\partial_X^2(U^2+2UH(U)+H(U)^2)=0.
\end{equation}
Rewriting \eqref{eq5} in conservation law form as 
\begin{eqnarray}
\partial_T U&=&\partial_X W , \label{eq5a}\\
\partial_T W&=&\partial_X (U+U^2+2UH(U)+H(U)^2) \label{eq5b}
\end{eqnarray}
yields Whitham's equations.

It is the  purpose of this paper to prove the following approximation result.
\begin{theorem} \label{theorem1}
There exist  $ C_1 > 0 $, $\varepsilon_0> 0 $ and  $C_2>0$ such that the following holds.
Let $U\in C([0,T_0],H^5(\mathbb{R},\mathbb{R}))$ be a solution of  \eqref{eq5}
with $
\sup_{T\in [0,T_0]} \| U(\cdot,T) \|_{H^5} \leq C_1  $ and let $V=H(U)$.
Then for all $\varepsilon \in(0,\varepsilon_0)$ we have solutions $(u,v)$ of  \eqref{eq1}-\eqref{eq2} such that
\begin{equation} \label{errestthm1}
\sup\limits_{t\in [0,T_0/ \varepsilon]} \sup \limits_{x\in \mathbb{R}}|(u,v)(x,t)-(U,V)(\varepsilon x,\varepsilon t)|\le C_2\varepsilon^{3/2}.
\end{equation}
\end{theorem}

\begin{remark}{\rm 
The $H^5$-control of $U$ is needed to estimate the residual generated by an improved approximation ansatz in $H^1$, see Lemma \ref{lemma1}. More generally, if we have  $U\in C([0,T_0],H^{s+4}(\mathbb{R},\mathbb{R}))$ 
with $
\sup_{T\in [0,T_0]} \| U(\cdot,T) \|_{H^{s+4}} \leq C_1  $ for $s\geq 1$, then our proof of Theorem \ref{theorem1} yields the error estimate
\[
\sup\limits_{t\in [0,T_0/ \varepsilon]} \|(u,v)(\cdot,t)-(U,V)(\varepsilon \,\cdot,\varepsilon t)\|_{H^{s}} \le C_2\varepsilon^{3/2}.
\] }
\end{remark} 

The question occurs whether to a given initial condition 
the associated solution of the KGB system can be approximated 
by a Whitham approximation. This will be discussed in Section \ref{secdiscussion}.

In order to explain why infinitely many normal from transformations 
have to be performed in the proof of Theorem \ref{theorem1}  we explain 
the strategy of our proof in more detail.
We write system \eqref{eq1}-\eqref{eq2} as a first order evolutionary system 
of the form 
$$ 
\partial_t \mathcal{W} = \Lambda \mathcal{W} + B(\mathcal{W},\mathcal{W}), 
$$ 
where $ \Lambda $ is a linear skew symmetric operator and $ B $ is a bilinear symmetric mapping.
By adding higher order terms to the approximation \eqref{eq4} we construct in  Section \ref{sec2} an approximation $ \psi $ which is 
$ \mathcal{O}(\varepsilon^2)$-close to $(\psi_u^{\text{Whitham}}, \psi_v^{\text{Whitham}})$ and satisfies formally 
$$
\text{Res}(\psi)= - \partial_t \psi + \Lambda \psi + B(\psi,\psi) =  \mathcal{O}(\varepsilon^4).
$$ 
The error function $ R $ defined by $ \mathcal{W}(x,t) = \psi( x, t) + \varepsilon^{\beta} R(x,t) $ fulfills
$$ 
\partial_t R = \Lambda R + 2 B(\psi,R) +  \varepsilon^{\beta} B(R,R) +  \varepsilon^{-\beta} \text{Res}(\psi).
$$ 
To prove an $ \mathcal{O}(1)$-bound for $ R $ on an $ \mathcal{O}(1/\varepsilon)$-time scale we have to control the terms on the right hand side on this long time scale.
The  first term is skew-symmetric and will lead to oscillations without any growth rates.
The last term can be $ \mathcal{O}(\varepsilon)$-bounded if $ \beta \leq 3 $. If $ \beta $ is chosen larger than $ 1 $ 
 the third terms gives a bound smaller than $ \mathcal{O}(\varepsilon)$.
However, the second term $ 2 B(\psi,R)  $ is only $ \mathcal{O}(1)$-bounded and can lead to $ \mathcal{O}(1)$ growth rates. In order to show that an $ \mathcal{O}(e^{t})$ growth  does not occur we use normal form transformations and energy estimates. Normal form transformations are near identity changes of variables of the form
\[
{R}_1=R+M(\psi,R)
\]
where  $M$ is a suitably chosen bilinear mapping. Eliminating the term $B(\psi,R)$ by such a normal form transformation is only  possible if the non-resonance condition 
$$
\inf_{j,n \in \{\pm1,\pm2\}, k \in \R} |\omega_j(k) - \omega_1(0) - \omega_n(k)| > 0
$$
is satisfied, which is not the case for the KGB model since $ \omega_1(0) = 0 $. But since the less restrictive non-resonance condition 
\begin{equation} \label{nonres2}
\inf_{k \in \R} \, \big\{ |\omega_{\pm1}(k) - \omega_1(0) - \omega_{\pm2}(k)|, |\omega_{\pm2}(k) - \omega_1(0) - \omega_{\pm1}(k)| \big\}  > 0
\end{equation}
is true for the KGB model the term $B(\psi,R)$ can be split into a resonant and a non-resonant part, i.e.,
$$ 
B(\psi,R) = B_{res}(\psi,R) + B_{non}(\psi,R)\,,
$$ 
and the non-resonant part $B_{non}(\psi,R)$ can be eliminated by a normal form transformation. After the normal form transformation in the equation for the new error function ${R}_1$  new terms of order $ \mathcal{O}(1) $
appear. They can be split again   into resonant and 
non-resonant terms. Another normal form transformation is necessary to eliminate these new non-resonant terms, but again 
terms of $ \mathcal{O}(1) $ are created. However, they are cubic w.r.t.~$ \psi$.
This goes ad infinitum and so the convergence of the composition of these infinitely many transformations
has to be proven. Since the $ n $-th transformation is of order $ \mathcal{O}(\| \psi \|^n) $
the convergence finally can be established with the help of the geometric series for $ \| \psi \| = \mathcal{O}(1) $, but sufficiently small w.r.t.~some $\Vert \cdot \Vert$-norm.
After all these transformations the equation for the transformed error $R_{\infty}$ takes the form 
  $$ 
\partial_t  R_{\infty}   = \Lambda R_{\infty} + F(\psi,R_{\infty}) +  \mathcal{O}(\varepsilon)
$$  
where $ F $ is a function which is linear w.r.t.~$ R_{\infty} $ and which contains infinitely many resonant terms.
Since all these terms now in contrast to the original $ B(\psi,R) $
have a long-wave character w.r.t.~$ t $, i.e., these terms depend explicitly  only on $\varepsilon t$ and not 
on $ t $, a suitably chosen energy $ E(R_{\infty}) $ satisfies 
  $$ 
\frac{d}{dt} E(R_{\infty})  =  \mathcal{O}(\varepsilon),
$$ 
and so an $  \mathcal{O}(1) $-bound for the transformed error $ R_{\infty} $ and the original error $R$, respectively, can be established on the 
$ \mathcal{O}(\varepsilon^{-1})$-time scale by applying Gronwall's inequality. The series of normal form transformations  can be found in Section 
\ref{sec3} and the energy estimates in Section \ref{sec4}.

We close this introduction with a number of remarks 
\begin{remark}{\rm 
Proving Theorem \ref{theorem1} is a nontrivial task since 
we have to prove an $ \mathcal{O}(1)$-bound for the error on an $ \mathcal{O}(1/\varepsilon)$-time scale. 
There exists a number of counter examples \cite{Schn96,SSZ14} where a formally derived amplitude equation 
makes wrong predictions about the dynamics of the original system.
}
\end{remark}

\begin{remark}{\rm 
Whitham's equations belong to the class of generic and universal amplitude 
equations containing the KdV
equation, the NLS equation, the Ginzburg-Landau equation, Burgers' equation, and so-called phase diffusion equations. 
Amplitude equations play an important role in the
description of spatially extended conservative or dissipative physical
systems, where they can be formally derived with the help of a multiple scaling ansatz. Whitham's equations are especially used to describe slow modulations in time and space of periodic wave trains in dispersive wave systems.
There exist  a series of approximation results for the
Ginzburg-Landau approximation, for instance in 
\cite{CE90,vH91,Sch94a,Sch94b}, for the 
KdV  approximation, for instance in  \cite{Cr85,SW00,SW02,D12}, and 
for  the NLS approximation, for instance in 
\cite{Kal87,Sch98,Sch05NLS,BSTU06,DS06,TW11,DSW12}. 
Approximation results for so-called phase diffusion
equations, Burgers equation or conservation laws describing modulations of periodic  waves in dissipative systems 
can be found in \cite{MS04b,MS04a,DSSS09}.
In the
conservative case, i.e. for Whitham's equations, the first nonlinear 
approximation result has been established in \cite{SD09}, namely the validity of  
Whitham's equations for the NLS equation as original system.  However, the spectral picture 
of the problem considered in \cite{SD09}
is as drawn in Figure \ref{figure1neu} and not as in Figure \ref{figure1.1}.}
\end{remark}
\begin{remark}{\rm 
Whitham derived his equations first in \cite{Wh65a,Wh65b} and they are
still a subject of active research, cf. \cite{DHM06,M06}. 
They  are an amplitude system for which so far there has not been established a satisfying theory which shows mathematically rigorously
that the original system behaves approximately as predicted by the associated amplitude equation.}
\end{remark}
\begin{remark}{\rm 
The Boussinesq equation \eqref{eq1} is a model equation for the water wave problem, whereas the solution of the Klein-Gordon equation \eqref{eq2} represents a scalar quantum field. Here we ignore this origin and couple them solely with the goal to obtain a spectral picture 
as plotted in Figure \ref{figure1.1}.}
\end{remark}
\begin{remark}{\rm
More generally, Whitham's equations are universal approximation equations for large classes of nonlinear PDEs of periodic wave type, see for example \cite{SD09}. Very often they are derived from the Lagrangian of the underlying problem
leading to a system of conservation laws, similar to  \eqref{eq5a}-\eqref{eq5b}. As mentioned above, the resonance structure of the KGB model is the same as
in the situation in which one is really interested in, namely the description of slow modulations in time and space of a periodic traveling wave in a dispersive wave system. By linearizing around the periodic wave in a co-moving frame, we obtain an eigenvalue problem which is periodic in the spatial variable. Its solutions are given by Bloch modes $e^{{\rm{i}}lx+{\rm{i}}\omega_n(l)t}v_n(l,x)$ with $n\in \mathbb{Z}\setminus \{0\}$, $l\in[\frac{-1}{2L},\frac{1}{2L})$, and where  $v_n$ possesses the same periodicity $L$ w.r.t. $x$ as the periodic wave. The curves $l\mapsto\omega_n(l)$ are ordered by $\omega_n(l)\le \omega_{n+1}(l)$ and by $\omega_n(l)=-\omega_{-n}(l)$. In general we have $\omega_{\pm 1}(0)=0$ and $\omega_{\pm 2}(0)\neq 0$ since in such systems the periodic wave is accompanied by an at least two-dimensional family of periodic waves. Whitham's equations describe the dynamics of the modes associated with the two curves $\omega_{\pm 1}$ in the  limit $l \rightarrow 0$, see Figure \ref{figure1.1}.
}\end{remark}
\begin{remark}{\rm \label{rem6}
As explained above we think that our analysis is a necessary step for 
the validity of Whitham's equations in the general situation.
However, before applying these ideas a number of additional questions have 
to be answered, most essential:
how to extract the wave numbers in non $ S^1 $-symmetric systems such that these satisfy 
equations which are suitable for existing functional analytic tools?
}\end{remark}
\begin{remark}{\rm \label{rem7}
Recently Whitham's equations have been in the focus of investigations concerning  
modulations of periodic wave trains in 
dissipative systems containing conservation laws \cite{JNRZ14}.
The problems with quadratic resonances addressed in the present work do not appear in the dissipative situation. 
We expect that 
the analysis for a justification result in the sense of Theorem \ref{theorem1} in the dissipative situation is  very similar to the  
one given in \cite[Section 6]{DSSS09} where a single conservation law has been justified 
as an amplitude equation.
}\end{remark}
\begin{remark}{\rm 
We expect that an underlying Hamiltonian structure 
may allow to find an energy which allows 
to perform a justification analysis of the Whitham approximation without 
the need of  infinitely many transformations. However, such an approach 
strongly depends on the problem. In contrast, the method presented here is rather
independent of the special original system 
in this class of problems
and hence more robust.
}\end{remark}

\noindent \textbf{Notation.} Possibly different constants that can be chosen 
independent of $ 0 < \varepsilon \ll1 $ are denoted by the same symbol $ C $.  From now on we write $\int$ instead of $\int_{-\infty}^{\infty}$. 
The space $ H^s_m $ consists of $ s $-times weakly
differentiable functions for which 
$\| u \|_{H^s_m} = \| u \rho^m \|_{H^s} = (\sum_{j= 0}^s \int | \partial_x^j
(u \rho^m) |^2(x) dx)^{1/2}$, with $ \rho(x) = \sqrt{1+x^2} $, is finite, 
 where we do not distinguish between scalar and vector-valued 
functions or real- and complex-valued functions.  We use $ H^s $ as an abbreviation 
for $ H^s_0 $. Moreover, we use the space $L^1_m$ with the norm $\| u \|_{L^1_m} = \| u \rho^m \|_{L^1}$.    The Fourier transform
of a function $u$ is denoted by 
$$(\mathcal{F}u)(k) = \widehat{u}(k) = \frac{1}{2 \pi} \int u(x) e^{-i k x} dx$$
and is an isomorphism between $ H^s_m $ and $ H^m_s $.
The point-wise multiplication $(uv)(x) = u(x)v(x)$ in $x$-space corresponds to the
convolution
$$ (\widehat{u}*\widehat{v})(k) = \int \widehat{u}(k-l)\widehat{v}(l) dl $$
in Fourier space. 
\medskip

\textbf{Acknowledgments:} The paper is partially supported
by the Deut\-sche Forschungsgemeinschaft DFG under the grant Schn520/9. The authors thank Mariana Haragus, James Kennedy and an unknown referee for their useful comments.

\section[The improved approximation]{The improved approximation and estimates for the residual}

\label{sec2}

As explained above we need  the residual to be small. With 
the approximation defined in \eqref{eq4} we formally find that 
$
\text{Res}_u(\psi_u^{\text{Whitham}}, \psi_v^{\text{Whitham}}) = \mathcal{O}(\varepsilon^4) $, but only $ \text{Res}_v(\psi_u^{\text{Whitham}}, \psi_v^{\text{Whitham}}) = \mathcal{O}(\varepsilon^2)
$.
In order to have $ \text{Res}_v= \mathcal{O}(\varepsilon^4) $, too, we extend the  ansatz  \eqref{eq4}  to 
\begin{equation}\label{ansatz}
\psi_u(x,t)=U(\varepsilon x,\varepsilon t)  \quad \text{and} \quad 
\psi_v(x,t)=V(\varepsilon x,\varepsilon t)+\varepsilon^2V_2(\varepsilon x,\varepsilon t)\,.
\end{equation}
We find
\begin{eqnarray*}
\text{Res}_v(\psi_u, \psi_v)
&=&-V-(U^2+2UV+V^2)+\varepsilon^2(-\partial_T^2V+\partial_X^2V-V_2-2UV_2-2VV_2)\\
& & +\varepsilon^4(-\partial_T^2V_2+V_2^2+\partial_X^2V_2).
\end{eqnarray*}
We formally obtain $ \text{Res}_v(\psi_u, \psi_v) = \mathcal{O}(\varepsilon^4) $ by choosing 
\begin{equation} \label{eq7}
V_2=\frac{\partial_X^2V-\partial_T^2V}{1+2U+2V}.
\end{equation}
For  $U$ and $V$ sufficiently small, but still of order $ \mathcal{O}(1) $, 
the function  $V_2$ is well-defined.
\begin{remark} {\rm In the following we estimate the difference between a true solution of 
\eqref{eq1}-\eqref{eq2} and the improved approximation defined in  \eqref{ansatz}. The estimate for the difference between a true solution of 
\eqref{eq1}-\eqref{eq2} and the original approximation defined in \eqref{eq4} then follows by the triangle inequality using
$$
\sup\limits_{t\in [0,T_0/ \varepsilon]} \sup \limits_{x\in \mathbb{R}}|(\psi_u,\psi_v)(x,t)-(\psi_u^{\text{Whitham}}, \psi_v^{\text{Whitham}})(x,t)|\le C\varepsilon^{2}.
$$}
\end{remark} 
The difference between a true solution of 
\eqref{eq1}-\eqref{eq2} and the improved approximation 
defines the error functions $R_u$ and $R_v$ by 
$$
\varepsilon^{\beta}R_u=u-\psi_u 
\quad \text{and} \quad \varepsilon^{\beta}R_v=v-\psi_v
$$ 
with a suitably chosen $\beta $. 
The  error functions satisfy 
\begin{eqnarray} \label{fly1}
\partial_t^2R_u&=&\partial_x^2R_u+\partial_t^2\partial_x^2R_u+2 \partial_x^2(\psi_uR_u+\psi_vR_u+\psi_uR_v+\psi_vR_v) \label{errorcalcu}\\ &&\qquad +\varepsilon^{\beta}\partial_x^2(R_u^2+2R_uR_v+R_v^2)\nonumber\\ &&\qquad \qquad +\varepsilon^{-\beta}\underbrace{\left(-\partial_t^2\psi_u+\partial_x^2\psi_u+\partial_t^2\partial_x^2\psi_u+\partial_x^2(\psi_u^2+2\psi_u\psi_v+\psi_v^2)\right)}_{=\text{Res}_u(\psi_u,\psi_v)}\,,\nonumber\\ \label{fly2}
\partial_t^2R_v&=&\partial_x^2R_v-R_v- 2(R_u\psi_u+R_u\psi_v+R_v\psi_u+R_v\psi_v ) \label{errorcalcv} \\ && \qquad-  \varepsilon^{\beta}(R_u^2+2R_uR_v+R_v^2)\nonumber \\
&&\qquad \qquad +\varepsilon^{-\beta}\underbrace{(-\partial_t^2\psi_v+\partial_x^2\psi_v-\psi_v- (\psi_u^2+2\psi_u\psi_v+\psi_v^2))}_{=\text{Res}_v(\psi_u,\psi_v)}\,,\nonumber
\end{eqnarray}
where the residual terms are formally of order 
$ \mathcal{O}(\varepsilon^4)$.
These equations will be solved in some Sobolev spaces.
Estimating the residual terms in these Sobolev spaces will lose $\varepsilon^{-1/2}$
due to the scaling properties of the $L^2-$norm, namely 
\begin{equation} \label{scaling}
\left(\int{|U(\varepsilon x)|^2dx}\right)^{1/2}=\left(\varepsilon^{-1}\int|U(X)|^2dX\right)^{1/2}, 
\end{equation}
and so we have the following lemma.

\begin{lemma} \label{lemma1}
For $s\ge 1$ there exist  $ C_1 > 0 $, $\varepsilon_0> 0 $ and  $C_2>0$ such that the following holds.
Let $U\in C([0,T_0],H^{s+4}(\mathbb{R},\mathbb{R}))$ be a solution of  \eqref{eq5}
with $
\sup_{T\in [0,T_0]} \| U(\cdot,T) \|_{H^{s+4}}$ $\leq C_1  $, let $V=H(U)$, and let $ V_2 $ be defined 
in \eqref{eq7}.
Then  for all $\varepsilon\in (0,\varepsilon_0)$ we have 
\[
\sup\limits_{t\in[0,T_0/\varepsilon]}(\|\text{\rm Res}_u(\psi_u,\psi_v)\|_{H^s} +\|\text{\rm Res}_v(\psi_u,\psi_v)\|_{H^s})< C_2 \varepsilon^{7/2}.
\]
\end{lemma}
{\bf Proof.}  Combining the formal calculations from above with the scaling properties \eqref{scaling} 
of the $ L^2 $-norm yields the required estimates. In order to avoid losing more powers of $ \varepsilon $ 
in products arising in $\text{\rm Res}_{u,v}$ only one factor is estimated in $ H^s $. All others are estimated in $ C^s_b $.
The assumption $U(\cdot,T)\in H^{s+4}(\mathbb{R},\mathbb{R})$ is necessary to estimate $\partial_X^2V_2\in H^s(\mathbb{R},\mathbb{R})$ via $V_2=\mathcal{O}(\partial_X^2V)$ due to \eqref{eq7}.  \qed 
\medskip

When writing \eqref{errorcalcu}-\eqref{errorcalcv} as a first order system we additionally need

\begin{lemma} \label{lemma1b}
Under the assumptions of Lemma \ref{lemma1}
 for all $\varepsilon\in (0,\varepsilon_0)$ we have 
\[
\sup\limits_{t\in[0,T_0/\varepsilon]}(\|\omega_1^{-1} \text{\rm Res}_u(\psi_u,\psi_v)\|_{H^s} +\|\omega_2^{-1} \text{\rm Res}_v(\psi_u,\psi_v)\|_{H^s})< C_2\varepsilon^{5/2}.
\]
\end{lemma}
{\bf Proof.} 
The terms of the residual have either spatial derivatives in front or are time derivatives which can be expressed via \eqref{eq5} as terms with spatial derivatives in front. Hence, in Fourier space all terms of the residual have at least 
a factor $ k $ and so the application of $ \omega_1(k)^{-1} $ to these terms is well-defined.
However, because of the long-wave character of the ansatz \eqref{ansatz} there is a loss of $ \mathcal{O}(\varepsilon^{-1}) $ since one derivative is 
canceled by the application of $ \omega_1(k)^{-1} $. Hence, the assertion of the lemma follows from Lemma \ref{lemma1}.
 \qed

\section{The series of normal form transformations}

\label{sec3}

In order to establish the validity of Theorem \ref{theorem1}
we have to prove an $ \mathcal{O}(1)$-bound for $ R_u $ and $ R_v $ on an $ \mathcal{O}(\varepsilon^{-1})$ 
time scale. Therefore we need to control the terms on the right hand sides of \eqref{fly1} and \eqref{fly2} on this long time scale. As already said, the main part of this paper is devoted to the handling of the linear $ \psi $-dependent  terms.
Therefore, the   error equations are rewritten in  the form
\begin{align}
\partial_t^2R_u&=\partial_x^2R_u+\partial_t^2\partial_x^2R_u+2 \partial_x^2(\psi_uR_u+\psi_vR_u+\psi_uR_v+\psi_vR_v)+\varepsilon p_{u,1}, \label{eq8}\\
\partial_t^2R_v&=\partial_x^2R_v-2 R_v- 2(\psi_uR_u+2\psi_vR_u+ \psi_uR_v+\psi_vR_v)+ \varepsilon p_{v,1}, \label{eq9}  
\end{align}
where the terms $p_{u,1}$ and $p_{v,1}$ are defined by
\begin{align*}
\varepsilon p_{u,1}&=\varepsilon^{\beta}\partial_x^2(R_u^2+2R_uR_v+R_v^2) +\varepsilon^{-\beta}\text{Res}_u,\\
\varepsilon p_{v,1}&=- \varepsilon^{\beta}(R_u^2+2R_uR_v+R_v^2)
+\varepsilon^{-\beta}\text{Res}_v.
\end{align*}
 These last  terms  provide high enough orders w.r.t.~$\varepsilon$ such that
they cause no difficulties in arriving at the $ \mathcal{O}(\varepsilon^{-1})$ time scale if we choose $\beta = 3/2$.
Because of Lemma~\ref{lemma1} we have 
\begin{equation*} \label{estimate}
\|\varepsilon p_{u,1} \|_{H^{s}}+ \|\varepsilon p_{v,1} \|_{H^{s+2}} \leq C ( \varepsilon^{3/2} (\| R_u \|_{H^{s+2}} +\| R_v \|_{H^{s+2}})^2 + \varepsilon^2).
\end{equation*}
We write \eqref{eq8}-\eqref{eq9} as a first order system,
which in Fourier space has the form
\begin{eqnarray} \label{italia1}
\partial_t\widehat{R}_u&=&i \omega_1\widehat{W}_u ,\\ \nonumber
\partial_t\widehat{W}_u&=&i \omega_1\widehat{R}_u+2i \omega_1(\widehat{\psi}_u*\widehat{R}_u+\widehat{\psi}_v*\widehat{R}_u+{\widehat{\psi}_u*\widehat{R}_v+\widehat{\psi}_v*\widehat{R}_v})+\varepsilon \widehat{p}_{u,2} ,\\
\partial_t\widehat{R}_v&=&i \omega_2\widehat{W}_v ,\label{italia2} \\
\partial_t\widehat{W}_v&=&i \omega_2\widehat{R}_v+  2 i \omega_2^{-1}({\widehat{\psi}_u*\widehat{R}_u+\widehat{\psi}_v*\widehat{R}_u}+{{\widehat{\psi}_u*\widehat{R}_v+\widehat{\psi}_v*\widehat{R}_v}})+\varepsilon \widehat{p}_{v,2} ,\nonumber
\end{eqnarray}
with  $ \widehat{p}_{u,2}(k,t) = - \varepsilon^{-1}i \omega_1^{-1}(k)\frac{1}{k^2+1}  \widehat{p}_{u,1}(k,t)  $ and 
$  \widehat{p}_{v,2}(k,t) =- \varepsilon^{-1} i  \omega_2^{-1}(k)  \widehat{p}_{v,1}(k,t)  $, 
where $ \widehat{p}_{u,j} $ and $ \widehat{p}_{v,j} $ are the Fourier transform of $ p_{u,j} $ and $ p_{v,j} $. 
Since
the nonlinear terms in \eqref{errorcalcu}  have two spatial derivatives in front, 
in Fourier space they are 
$ \mathcal{O}(k^2) $, and so 
the application of $ \omega_1(k)^{-1} $ is well-defined for all terms containing $ \widehat{R}_u $ and $ \widehat{R}_v $. Since the $ \varepsilon $-order of $ \widehat{R}_u $ and $ \widehat{R}_v $, in contrast to the residual terms,  purely comes from the amplitude and not from the 
long-wave character of the ansatz \eqref{ansatz} the application of $ \omega_1(k)^{-1} $ causes no loss of $ \mathcal{O}(\varepsilon^{-1}) $ for all terms containing $ \widehat{R}_u $ and $ \widehat{R}_v $. For the terms coming from the residual we now use Lemma \ref{lemma1b} instead of Lemma~\ref{lemma1} and obtain
$$ 
\| \varepsilon  \widehat{p}_{u,2} \|_{H^0_s} + \| \varepsilon  \widehat{p}_{v,2} \|_{H^0_s} \leq C(  \varepsilon^{3/2} (\| \widehat{R}_u \|_{H^0_s} +\| \widehat{R}_v \|_{H^0_s})^2  +\varepsilon).
$$
We diagonalize \eqref{italia1}-\eqref{italia2} with 
\begin{equation} \label{diagonali}
 \left(
\begin{array}{c} \widehat{R}_u \\ \widehat{W}_u 
\end{array}
\right) = 
\frac{1}{\sqrt{2}}\left(
\begin{array}{cc}
  1 &  1   \\
 1 & -1    
\end{array}
\right)
\left(
\begin{array}{c} \widehat{R}_1 \\ \widehat{R}_{-1} 
\end{array}
\right),  
 \left(
\begin{array}{c} \widehat{R}_v \\ \widehat{W}_v 
\end{array}
\right) = 
\frac{1}{\sqrt{2}}\left(
\begin{array}{cc}
  1 &  1   \\
 1 & -1    
\end{array}
\right)
\left(
\begin{array}{c} \widehat{R}_2 \\ \widehat{R}_{-2} 
\end{array}
\right) 
 \end{equation}
and find 
\begin{eqnarray} \label{thunder1}
\partial_t\widehat{R}_1&=& i\omega_1\widehat{R}_1+i\omega_1(S_1(\Psi,\widehat{R}_{\pm 1}) + {S_2(\Psi,\widehat{R}_{\pm 2})}
)+\varepsilon \widehat{p}_1,
\\ 
\partial_t\widehat{R}_{-1} &=&-i\omega_1\widehat{R}_{-1}-i\omega_1(S_1(\Psi,\widehat{R}_{\pm 1}) + {S_2(\Psi,\widehat{R}_{\pm 2})})+\varepsilon \widehat{p}_{-1}, 
\\
\partial_t\widehat{R}_2&=& i\omega_2\widehat{R}_2+i\omega_2^{-1}({S_1(\Psi,\widehat{R}_{\pm 1})} + S_2(\Psi,\widehat{R}_{\pm 2}))+\varepsilon \widehat{p}_2, 
\\ \partial_t\widehat{R}_{-2}&=&-i\omega_2\widehat{R}_{-2}-i\omega_2^{-1}({S_1(\Psi,\widehat{R}_{\pm 1})} + S_2(\Psi,\widehat{R}_{\pm 2}))+\varepsilon \widehat{p}_{-2},\label{thunder4}
\end{eqnarray}
where  $ \Psi  ={\psi}_u+{\psi}_v $, 
$ S_1(\Psi,\widehat{R}_{\pm 1}) = \widehat{\Psi}*(\widehat{R}_1+\widehat{R}_{-1}) $,
and
$S_2(\Psi,\widehat{R}_{\pm 2}) = \widehat{\Psi}*(\widehat{R}_2+\widehat{R}_{-2})$.
 The terms $  \widehat{p}_{w}$ with $w \in \{\pm1,\pm2\}$ can be estimated by 
\begin{equation*} 
 \| \varepsilon  \widehat{p}_{w} \|_{H^0_s} \leq C (  \varepsilon^{3/2} (\| \widehat{R}_{-2} \|_{H^0_s} +\ldots +\| \widehat{R}_2 \|_{H^0_s})^2 + \varepsilon) .
\end{equation*}

\subsection{The first normal form transformation}

From the term $ S_1 $ in the equation for $ \widehat{R}_{\pm 1}$ 
we already know from \eqref{bousseq1} that it  
can be estimated with the help of energy estimates.
The main observation of \cite{ChSch11} was that the long wave character w.r.t. time of 
the term $ S_2 $ in the equation 
for $ \widehat{R}_{\pm 2} $ can be used to construct an energy which allows to get rid of this term, too.

 We try to eliminate the other terms with the help of normal form transformations.
In order to do so 
we set $\widehat{R}_{l,1}= \widehat{R}_{l}$ for  $l \in\{\pm1,\pm2\} $ and make the ansatz 
\begin{eqnarray}
 \widehat{R}_{1,2} & = & \widehat{R}_{1,1}+  M_{1}^{(1)}(\Psi,\widehat{R}_{\pm 2 ,1})  \label{trafo1a} ,\\
\widehat{R}_{2,2} & = & \widehat{R}_{2,1}+   M_{2}^{(1)}(\Psi,\widehat{R}_{\pm 1,1})
 \label{trafo2a} ,\\
\widehat{R}_{-1,2} & = & \widehat{R}_{-1,1}+  M_{-1}^{(1)}(\Psi,\widehat{R}_{\pm 2 ,1})  \label{trafo-1a} ,\\
\widehat{R}_{-2,2} & = & \widehat{R}_{-2,1}+   M_{-2}^{(1)}(\Psi,\widehat{R}_{\pm 1,1})
 \label{trafo-2a}, 
\end{eqnarray}
with $ M_{1}^{(1)} $ and  $ M_{-1}^{(1)} $ linear in $ \widehat{R}_{\pm 2,1} $
as well as $ M_{2}^{(1)}$ and $ M_{-2}^{(1)}$ linear in $ \widehat{R}_{\pm 1,1} $, where $\widehat{R}_{\pm n,j}= ( \widehat{R}_{-n,j},\widehat{R}_{n,j})$.

Since $S_2$ is of the form 
$$
S_{2}(\Psi,\widehat{R}_{\pm 2 ,1}) =
  \sum_{l\in\{2,-2\}} S_{2l}(\Psi,\widehat{R}_{l,1})  ,
$$
with 
$$
S_{2l}(\Psi,\widehat{R}_{l,1}) =
\int \widehat{s}_{2l}(k,k-m,m)
\widehat{\Psi}(k-m,\varepsilon  t) \widehat{R}_{l,1}(m,t) dm ,
$$
we set 
$$
M_{1}^{(1)}(\Psi,\widehat{R}_{\pm 2 ,1}) =
  \sum_{l\in\{2,-2\}} M_{1l}^{(1)}(\Psi,\widehat{R}_{l,1})  ,
$$
with 
$$
M_{1l}^{(1)}(\Psi,\widehat{R}_{l,1}) =
\int \widehat{m}_{1l}^{(1)}(k,k-m,m)
\widehat{\Psi}(k-m,\varepsilon  t) \widehat{R}_{l,1}(m,t) dm.
$$
Using $ \partial_t \Psi = \mathcal{O}(\varepsilon) $ we find 
\begin{eqnarray*}
\partial_t\widehat{R}_{1,2}&=& \partial_t \widehat{R}_{1,1}+ \sum_{l\in\{2,-2\}} M_{1l}^{(1)}(\Psi,\partial_t \widehat{R}_{l ,1}) +\mathcal{O}(\varepsilon) \\
& = &  i\omega_1\widehat{R}_{1,1}+i\omega_1 S_1(\Psi,\widehat{R}_{\pm 1,1}) 
+ i\omega_1 \sum_{l\in\{2,-2\}}  {S_{2l} (\Psi,\widehat{R}_{l,1})} \\ & & + \sum_{l\in\{2,-2\}} M_{1l}^{(1)} \big(\Psi,i\omega_l\widehat{R}_{l,1}+i\omega_l^{-1}({S_1(\Psi,\widehat{R}_{\pm 1,1})} + S_2(\Psi,\widehat{R}_{\pm 2,1}))\big) \\ &&
+\mathcal{O}(\varepsilon)
\\
& = &  i\omega_1\widehat{R}_{1,2}- i\omega_1
\sum_{l\in\{2,-2\}} M_{1l}^{(1)}(\Psi,\widehat{R}_{l ,1}) \\&& 
+i\omega_1 S_1(\Psi,\widehat{R}_{\pm 1,1}) 
+ i\omega_1 \sum_{l\in\{2,-2\}}  {S_{2l} (\Psi,\widehat{R}_{l,1})} 
 \\& & + \sum_{l\in\{2,-2\}} M_{1l}^{(1)} \big(\Psi,i\omega_l\widehat{R}_{l,1}+i\omega_l^{-1}({S_1(\Psi,\widehat{R}_{\pm 1,1})} + S_2(\Psi,\widehat{R}_{\pm 2,1}))\big) \\ &&
+\mathcal{O}(\varepsilon).
\end{eqnarray*}
In order to eliminate the term 
$ S_{2l}(\Psi,\widehat{R}_{l,1})$  we have to choose 
\begin{eqnarray*}
- i\omega_1
M_{1l}^{(1)}(\Psi,\widehat{R}_{l ,1}) 
+M_{1l}^{(1)}(\Psi,i\omega_l\widehat{R}_{l,1})
+
i\omega_1 S_{2l}(\Psi,\widehat{R}_{l,1})=0
\end{eqnarray*}
and find
\begin{eqnarray*}
i (\omega_1(k)-\omega_l(m))
\widehat{m}_{1l}^{(1)}(k,k-m,m)
= i \omega_1(k) \widehat{s}_{2l}(k,k-m,m) 
\end{eqnarray*}
for $ l \in \{-2,2\} $.
This equation can be solved w.r.t. the kernel
$ \widehat{m}_{1l}^{(1)}(k,k-m,m)$
due to the validity of the 
non-resonance condition 
\begin{eqnarray*}
\inf_{k,l\in \mathbb{R}} |\omega_1(k)\pm \omega_2(m)| \geq  1.
\end{eqnarray*}
Similarly the equation for $ \widehat{R}_{2,2} $ can be handled. 
We find
\begin{eqnarray*}
i (\omega_2(k)-\omega_l(m))
\widehat{m}_{2l}^{(1)}(k,k-m,m)
= i \omega_2^{-1}(k) \widehat{s}_{1l}(k,k-m,m)
\end{eqnarray*}
for $ l \in \{-1,1\} $, where $\widehat{s}_{1l}$ is defined in an analogous way as $\widehat{s}_{2l}$.
This equation can be solved w.r.t. the kernel
$ \widehat{m}_{2l}^{(1)}(k,k-m,m)$
due to the validity of the 
non-resonance condition 
\begin{equation} \label{easton1}
\inf_{k,m\in \mathbb{R}} |\omega_2(k)\pm \omega_1(m)| \geq  1.
\end{equation}
Although valid
this non-resonance condition will be weakened 
in the following in order to present an approach which works for more general systems, too.
After the transform we obtain 
\begin{eqnarray*}
\partial_t\widehat{R}_{1,2}&=&   i\omega_1\widehat{R}_{1,2}
+i\omega_1 S_1(\Psi,\widehat{R}_{\pm 1,1})  \\& & + 
\sum_{l\in\{2,-2\}} M_{1l}^{(1)} \big(\Psi,i\omega_l^{-1}({S_1(\Psi,\widehat{R}_{\pm 1,1})} + S_2(\Psi,\widehat{R}_{\pm 2,1}))\big) 
+\mathcal{O}(\varepsilon) 
\end{eqnarray*}
and similarly for $ \widehat{R}_{2,2} $, $ \widehat{R}_{-1,2} $, and $ \widehat{R}_{-2,2} $. On the right hand side there are still terms $ \widehat{R}_{j,1} $.
In order to replace them with $ \widehat{R}_{j,2} $ terms we have to invert the above normal form transformation.
We write the inverse as 
\begin{eqnarray}
 \widehat{R}_{\pm 1,1} & = & \widehat{R}_{\pm 1,2}+  \widetilde{M}_{\pm 1}^{(1)}(\Psi,\widehat{R}_{,2})  \label{trafo1ab} ,\\
\widehat{R}_{\pm 2,1} & = & \widehat{R}_{\pm 2,2}+   \widetilde{M}_{\pm 2}^{(1)}(\Psi,\widehat{R}_{,2})
 \label{trafo2ab} ,
\end{eqnarray}
with $ \widetilde{M}_{l}^{(1)} $ linear in $ \widehat{R}_{,2} $, where
$\widehat{R}_{,j}= ( \widehat{R}_{-2,j},\widehat{R}_{-1,j},\widehat{R}_{1,j},\widehat{R}_{2,j})$.
Substituting $ \widehat{R}_{,1} $ in terms of $ \widehat{R}_{,2} $ finally yields
\begin{align*}
\partial_t\widehat{R}_{1,2} = \;&   i\omega_1\widehat{R}_{1,2}
+i\omega_1 S_1(\Psi,\widehat{R}_{\pm 1,2})  +i\omega_1 S_1(\Psi, \widetilde{M}_{\pm 1}^{(1)}(\Psi,\widehat{R}_{,2})  )
\\[2mm] & 
+ \sum_{l\in\{2,-2\}} \big( M_{1l}^{(1)}(\Psi,i\omega_l^{-1}{S_1(\Psi,\widehat{R}_{\pm 1,2})}) 
+ M_{1l}^{(1)}(\Psi,i\omega_l^{-1}{S_1(\Psi,  \widetilde{M}_{\pm 1}^{(1)}(\Psi,\widehat{R}_{,2})  )}) \big) 
\\& 
+ \sum_{l\in\{2,-2\}} \big( M_{1l}^{(1)}(\Psi,i\omega_l^{-1} S_2(\Psi,\widehat{R}_{\pm 2,2}))
+ M_{1l}^{(1)}(\Psi,i\omega_l^{-1} S_2(\Psi, \widetilde{M}_{\pm 2}^{(1)}(\Psi,\widehat{R}_{,2}))) \big)\\ &
+\mathcal{O}(\varepsilon) 
\end{align*}
and similarly for $ \widehat{R}_{2,2} $, $ \widehat{R}_{-1,2} $, and $ \widehat{R}_{-2,2} $.
Hence new terms of order $ \mathcal{O}(1) $ 
are created. 
However, there is only one term of order  $ \mathcal{O}(\| \Psi \|) $, namely $i\omega_1 S_1(\Psi,\widehat{R}_{\pm 1,2})$ from which we already know how to handle it via \eqref{bousseq1} with the help
of energy estimates. All other terms are of order $ \mathcal{O}(\| \Psi \|^2) $ or higher.
Some of them are resonant, but of long wave form and will be included into the energy estimates.

Some of them are non-resonant 
but not of long wave form.     
They will be eliminated by another normal form transform,
but ad infinitum by such transformations 
terms of order $ \mathcal{O}(1)$ are created.
Therefore, we have to prove the convergence 
of this procedure.
It is based on the fact that in the $ j $-th step only 
terms of order  $ \mathcal{O}(\| \Psi \|^j) $ or higher 
will be affected.
This will be discussed in detail in the 
subsequent sections.

\subsection{The recursion formulas}
In order to control the action of the infinitely many near-identity changes of variables which we will apply to \eqref{thunder1}-\eqref{thunder4} it is essential to extract the structure of this system. 
We will see that after performing  $j-1$ transformations the error equations will be of the form
\begin{eqnarray}\label{eq17} 
 \partial_t\widehat{R}_{1,j}(k,t)&=& i \omega_1(k)\widehat{R}_{1,j}(k,t) +\varepsilon\widehat{p}_{1,j}(k,t)  \\
 & & +   i \omega_1(k) \int \widehat{f}_{1,res}^{(j)}(k,k-m,\varepsilon  t)\big( \widehat{R}_{1,j}(m,t)+\widehat{R}_{-1,j}(m,t)\big)  dm\nonumber \\
 & & +  i \omega_1(k)  \int \widehat{f}_{1,non}^{(j)}(k,k-m,\varepsilon  t) \big(\widehat{R}_{2,j}(m,t)+\widehat{R}_{-2,j}(m,t)\big)  dm , \nonumber 
 \end{eqnarray}
 \begin{eqnarray}
\label{eq18}
\partial_t\widehat{R}_{2,j}(k,t)&=& i \omega_2(k)\widehat{R}_{2,j}(k,t) +\varepsilon \widehat{p}_{2,j}(k,t) \\
 & & +  i \omega_2^{-1}(k) \int \widehat{f}_{2,non}^{(j)}(k,k-m,\varepsilon  t) \big(\widehat{R}_{1,j}(m,t)+\widehat{R}_{-1,j}(m,t)\big)  dm \nonumber \\ 
&&  +  i \omega_2^{-1}(k) \int \widehat{f}_{2,res}^{(j)}(k,k-m,\varepsilon  t)( \widehat{R}_{2j}(m,t)+\widehat{R}_{-2,j}(m,t)\big)  dm , \nonumber 
\end{eqnarray}
and similarly for $ \partial_t\widehat{R}_{-1,j} $ and $ \partial_t\widehat{R}_{-2,j} $.
For  $l \in\{\pm1,\pm2\} $
we set   $\widehat{p}_{l,1}= \widehat{p}_{l}$
and find
\begin{equation} \label{eric}
\widehat{f}_{l,res}^{(1)}(k,k-m,\varepsilon t)=\widehat{f}_{l,non}^{(1)}(k,k-m,\varepsilon t)= \widehat{\Psi}(k-m,\varepsilon t)  .
\end{equation}
Since $\widehat{R}_{-l,j}$ will be the complex conjugate of $\widehat{R}_{l,j}$, we have 
$$\widehat{f}_{-l,res}^{(j)}=\overline{\widehat{f}_{l,res}^{(j)}},  \quad \widehat{f}_{-l,non}^{(j)}=\overline{\widehat{f}_{l,non}^{(j)}}, \quad \textrm{and}  \quad \widehat{p}_{-l,j}=\overline{\widehat{p}_{l,j}}$$  for $j \in \N$ and $l \in\{1,2\}$.  Hence, it is sufficient to analyze the equations for $\widehat{R}_{|l|,j}$.

We derive now recursion formulas for the terms 
$\widehat{p}_{l,j}$, $\widehat{f}_{l,res}^{(j)}$, and $\widehat{f}_{l,non}^{(j)}$.
In order to do so we introduce the 
$j$-th near identity change of variables by 
\begin{eqnarray}
 \widehat{R}_{1,j+1}(k,t) & = & \widehat{R}_{1,j}(k,t)+  \sum_{l\in\{2,-2\}}\int \widehat{g}_{1l}^{(j)}(k,k-m,\varepsilon  t) \widehat{R}_{l,j}(m,t) dm \label{trafo1} ,\\
\widehat{R}_{2,j+1}(k,t) & = & \widehat{R}_{2,j}(k,t)+  \sum_{l\in\{1,-1\}}\int \widehat{g}_{2l}^{(j)}(k,k-m,\varepsilon  t) \widehat{R}_{l,j}(m,t) dm \label{trafo2} .
\end{eqnarray}
The perturbation of the identity is chosen in accordance  with the validity of the 
non-resonance condition \eqref{nonres2}.
We assume for the moment that the transformation \eqref{trafo1}-\eqref{trafo2} is invertible and that its inverse has the form 
\begin{eqnarray}
 \widehat{R}_{i,j}(k,t)&=&\widehat{R}_{i,j+1}(k,t)+\sum_{l\in\{\pm 1, \pm 2\}}\int\widehat{h}_{i l}^{(j)}(k,k-m,\varepsilon t)\widehat{R}_{l,j+1}(m,t)dm \label{inverse} 
\end{eqnarray}
for $i\in\{\pm 1, \pm 2\}$. Differentiating \eqref{trafo1} w.r.t. time yields
\begin{eqnarray*}
 \partial_t\widehat{R}_{1,j+1}(k,t)&=&\partial_t\widehat{R}_{1,j}(k,t)+\sum_{l\in\{2,-2\}}\int\widehat{g}_{1l}^{(j)}(k,k-m,\varepsilon t)\partial_t\widehat{R}_{l,j}(m,t)dm\\ 
                                 &&+\varepsilon\sum_{l\in\{2,-2\}}\int \partial_T\widehat{g}_{1l}^{(j)}(k,k-m,\varepsilon  t) \widehat{R}_{l,j}(m,t) dm.
\end{eqnarray*}
Then we replace the $ \partial_t\widehat{R}_{l,j}(k,t) $ by the equations \eqref{eq17}-\eqref{eq18}. Finally, we replace the $ \widehat{R}_{l,j}(k,t) $ via \eqref{inverse}
by the $ \widehat{R}_{l,j+1}(k,t) $.
In order to eliminate for $l\in\{\pm2\}$ the non-resonant term 
$$
i\omega_1(k) \int \widehat{f}_{1,non}^{(j)}(k,k-m,\varepsilon  t) \widehat{R}_{l,j}(m,t) dm
$$
we could proceed as above and choose the functions $\widehat{g}_{1l}^{(j)}$
 to satisfy
\begin{align}\label{resulta}
 0=&-i\omega_1(k) \int \widehat{g}_{1l}^{(j)}(k,k-m,\varepsilon  t) \widehat{R}_{l,j}(m,t) dm \\
    &+\int i\omega_l(m)\widehat{g}_{1l}^{(j)}(k,k-m,\varepsilon t)\widehat{R}_{l,j}(m,t)dm \nonumber \\
   &+i\omega_1(k) \int \widehat{f}_{1,non}^{(j)}(k,k-m,\varepsilon  t) \widehat{R}_{l,j}(m,t) dm 
   \nonumber
\end{align}
or equivalently to satisfy
\[
 \widehat{g}_{1l}^{(j)}(k,k-m,\varepsilon t)=i\omega_1(k)(i\omega_1(k)-i\omega_l(m))^{-1}\widehat{f}_{1,non}^{(j)}(k,k-m,\varepsilon  t)\,.
\] 
However, for more  general systems the non-resonance condition \eqref{easton1}
will not be valid. Hence, in order to present an approach 
which works for more general systems, too, we proceed differently at this point.
Since $ \widehat{\Psi} $ is strongly concentrated at the wave number $ k = 0 $
the difference 
$$
\varepsilon  \widehat{r}_{1,j}(k,t) = \sum_{l\in\{2,-2\}} \int i(\omega_l(m)- \omega_l(k)) \widehat{g}_{1l}^{(j)}(k,k-m,\varepsilon t)\widehat{R}_{l,j}(m,t)dm 
$$
will be of order $ \mathcal{O}(\varepsilon) $.
This can be expected for $ j = 1 $ due to \eqref{eric} and will be proved subsequently by induction for all $ j \geq  1 $. By replacing \eqref{resulta} by 
\begin{align}\label{result}
 0=&-i\omega_1(k) \int \widehat{g}_{1l}^{(j)}(k,k-m,\varepsilon  t) \widehat{R}_{l,j}(m,t) dm\\
    &+\int i\omega_l(k)\widehat{g}_{1l}^{(j)}(k,k-m,\varepsilon t)\widehat{R}_{l,j}(m,t)dm \nonumber \\
   &+i\omega_1(k) \int \widehat{f}_{1,non}^{(j)}(k,k-m,\varepsilon  t) \widehat{R}_{l,j}(m,t) dm 
   \nonumber 
\end{align}
or equivalently 
\begin{equation} \label{2}
 \widehat{g}_{1l}^{(j)}(k,k-m,\varepsilon t)=i\omega_1(k)(i\omega_1(k)-i\omega_l(k))^{-1}\widehat{f}_{1,non}^{(j)}(k,k-m,\varepsilon  t)\,.
\end{equation}
we only made an error of order $ \mathcal{O}(\varepsilon) $
and come to the much weaker non-resonance condition $ \eqref{nonres2} $.

By a straightforward calculation we find 
\begin{align}\label{resonant}
 & \, \widehat{f}_{1,res}^{(j+1)}(k, k-m, \varepsilon t)- \widehat{f}_{1,res}^{(j)}(k, k-m, \varepsilon t)\\ =&\sum_{\lambda\in\{2,-2\}}\int \widehat{\widetilde{g}}_{1\lambda}^{(j)}(k, k-l, \varepsilon t) \omega_2^{-1}(l) \widehat{f}_{\lambda ,non}^{(j)}(l, l-m, \varepsilon t)dl\nonumber \\
                                        & +\sum_{\kappa\in\{1,-1\}} \int\widehat{f}_{1,res}^{(j)}(k, k-l, \varepsilon t)\widehat{h}_{\kappa \mu}^{(j)}(l, l-m, \varepsilon t)dl \nonumber\\
                                       & +\sum_{\lambda\in\{2,-2\}, \atop \kappa\in\{1,-1\}}
                                       \int\int\widehat{\widetilde{g}}_{1\lambda}^{(j)}(k, k-l_1, \varepsilon t)\omega_2^{-1}(l_1)\widehat{f}_{\lambda ,non}^{(j)}(l_1, l_1-l_2, \varepsilon t)\nonumber\widehat{h}_{\kappa\mu}^{(j)}(l_2, l_2-m, \varepsilon t)dl_2dl_1 \nonumber\\
                                        &+\sum_{\lambda\in\{2,-2\}, \atop \kappa\in\{2,-2\}}
                                        \int \int\widehat{\widetilde{g}}_{1\lambda}^{(j)}(k, k-l_1, \varepsilon t)\omega_2^{-1}(l_1)\widehat{f}_{\lambda ,res}^{(j)}(l_1, l_1-l_2, \varepsilon t)\nonumber
                                                                           \widehat{h}_{\kappa\mu}^{(j)}(l_2, l_2-m, \varepsilon t)dl_2dl_1,\nonumber
\end{align}
and 
\begin{align}
& \, \widehat{f}_{1,non}^{(j+1)}(k, k-m, \varepsilon t)\label{nonresonant} \\ = & \sum_{\lambda\in\{2,-2\}}\int \widehat{\widetilde{g}}_{1\lambda}^{(j)}(k, k-l, \varepsilon t)\omega_2^{-1}(l)\widehat{f}_{\lambda ,res}^{(j)}(l, l-m, \varepsilon t)dl\nonumber\\
                                       &+ \sum_{\kappa\in\{1,-1\}} \int\widehat{f}_{1,res}^{(j)}(k, k-l, \varepsilon t)\widehat{h}_{\kappa \mu}^{(j)}(l, l-m, \varepsilon t)dl \nonumber\\
                                        &+\sum_{\lambda\in\{2,-2\}, \atop \kappa\in\{1,-1\}}\int\int\widehat{\widetilde{g}}_{1\lambda}^{(j)}(k, k-l_1, \varepsilon t)\omega_2^{-1}(l_1)\widehat{f}_{\lambda ,non}^{(j)}(l_1, l_1-l_2, \varepsilon t)\nonumber\widehat{h}_{\kappa\mu}^{(j)}(l_2, l_2-m, \varepsilon t)dl_2dl_1 \nonumber\\
                                        &+\sum_{\lambda\in\{2,-2\}, \atop \kappa\in\{2,-2\}}\int \int\widehat{\widetilde{g}}_{1\lambda}^{(j)}(k, k-l_1, \varepsilon t)\omega_2^{-1}(l_1)\widehat{f}_{\lambda ,res}^{(j)}(l_1, l_1-l_2, \varepsilon t)\nonumber\widehat{h}_{\kappa\mu}^{(j)}(l_2, l_2-m, \varepsilon t)dl_2dl_1,\nonumber 
\end{align}
where we used the abbreviation 
$$ 
\widehat{\widetilde{g}}_{1\lambda}^{(j)}(k, k-l, \varepsilon t)=  (\omega_1(k))^{-1}\widehat{{g}}_{1\lambda}^{(j)}(k, k-l, \varepsilon t) = 
-i (i\omega_1(k)-i\omega_{\lambda}(k))^{-1}\widehat{f}_{1,non}^{(j)}(k,k-m,\varepsilon  t).
$$
Moreover, we have 
\begin{align}             \label{epsilon}
           \widehat{p}_{1,j+1}(k,t)- \widehat{p}_{1,j}(k,t)=& \sum_{\lambda\in\{2,-2\}}\int \partial_T\widehat{g}_{1\lambda}^{(j)}(k,k-m,\varepsilon  t) \widehat{R}_{\lambda,j}(m,t) dm \\
                  &
                  +\sum_{\lambda\in\{2,-2\}}\int\widehat{g}_{1\lambda}^{(j)}(k,k-m,\varepsilon t)\widehat{p}_{\lambda,j}(m,t)dm +  \widehat{r}_{1,j}(k,t).\nonumber
\end{align}    
Due to the symmetry in \eqref{eq17} and \eqref{eq18} we  obtain similar equations for $\widehat{R}_2$  but with the roles of $\widehat{R}_2$ and $\widehat{R}_1$ interchanged.

\subsection{The functional analytic set-up and the inversion of the normal form transformations}

In order to control the functions $ f $, $ g $, and $ h $, we introduce 
the norm
\begin{equation} \label{x-norm}
 \Vert f\Vert_{X^{s,\varepsilon}}:=\int \sup\limits_{k\in\mathbb{R}}|f(k,l)|(1+(l/ \varepsilon)^2)^{s/2}dl.
\end{equation}
This norm reflects that $ f $, $ g $, and $ h $ are (infinite) sums of terms 
$ \widehat{\kappa}^{(j)}(k)\, \varepsilon^{-1} \widehat{\varphi}^{(j)}\big(\frac{k-m}{\varepsilon},\varepsilon  t\big ) $,  where $\widehat{\kappa}^{(j)}$ is Lipschitz continuous and determined by $\omega_1, \omega_2$ and where  $\widehat{\varphi}^{(j)}(\cdot, \varepsilon  t)$ belongs to $L^{1}_{s+1}$ and is determined by $\widehat{\Psi}(\cdot, \varepsilon  t)$. 
With the help of Young's inequality for convolutions we have 
\begin{equation} \label{garfunkel}
\Vert \int\widehat{f}(k,k-m,\varepsilon  t) \widehat{R}(m,t) dm
\Vert_{H^0_s} \le C \Vert\widehat{f}\Vert_{X^{s,\varepsilon}} \Vert\widehat{R}\Vert_{H^0_s}.
\end{equation}
The following lemma  allows us to control the convolution of  $ f $, $ g $, and $ h $
in the previous recursion formulas.
\begin{lemma}\label{lemma11}
 For $ s > 0 $  the following estimate holds:
 $$\left\Vert\int_{\mathbb{R}} f(\cdot_1,\cdot_1-l)g(l,l-\cdot_2)dl\right\Vert_{X^{s,\varepsilon}}\le \Vert f\Vert_{X^{s,\varepsilon}}\Vert g\Vert_{X^{s,\varepsilon}}.$$
\end{lemma}
\begin{proof}
Using Young's inequality for convolutions in weighted $L^1$-spaces yields
 \begin{eqnarray*}
&& \int\sup\limits_{k\in\mathbb{R}}\big|\int f(k,k-l)g(l,l-m)dl\big|(1+((k-m)/ \varepsilon)^2)^{s/2}d(k-m)\\
&&\le\int \sup_{k \in \mathbb{R}}\int\sup\limits_{\widetilde{k}\in\mathbb{R}}|f(\widetilde{k},k-l)|\sup\limits_{\widetilde{k}\in\mathbb{R}}|g(\widetilde{k},l-k+m)|dl(1+(m/ \varepsilon)^2)^{s/2}dm \\
&&\le\int \sup_{k \in \mathbb{R}}\int\sup\limits_{\widetilde{k}\in\mathbb{R}}|f(\widetilde{k},l)|\sup\limits_{\widetilde{k}\in\mathbb{R}}|g(\widetilde{k},m-l)|dl(1+(m/ \varepsilon)^2)^{s/2}dm \\
&&\le C\int\sup\limits_{k\in\mathbb{R}}|f(k,l)|(1+(l/ \varepsilon)^2)^{s/2}dl\int\sup\limits_{k\in\mathbb{R}}|g(k,m)|(1+(m/ \varepsilon)^2)^{s/2}dm\\
&&=\Vert f(k,l)\Vert_{X^{s,\varepsilon}} \Vert g(k,l)\Vert_{X^{s,\varepsilon}}.
\end{eqnarray*}
\end{proof}
For 
$ \Vert\widehat{g}_{il}^{(j)}(\varepsilon t)\Vert_{X^{s,\varepsilon}} $
sufficiently small, but independent of $ 0 < \varepsilon \ll 1 $,  the transformation \eqref{trafo1}-\eqref{trafo2} is
invertible.
\begin{lemma} \label{lemma12}
For $\widehat{R}_{i,j}(t)\in H^0_s$ with  $s\ge1$ define $\widehat{R}_{i,j+1}(t)$ by 
\begin{equation}
 \widehat{R}_{j+1}(k,t)=(I+T^{(j)})(\widehat{R}_{j}(k,t)), \label{Trafo}
\end{equation}
where 
\begin{equation}
 T^{(j)}(\widehat{R}_j)=\left(\begin{array}{cccc} 0 & 0 & T_{12}^{(j)} & T_{1-2}^{(j)} \\ 0 & 0 & T_{-12}^{(j)} & T_{-1-2}^{(j)} 
                                                    \\ T_{21}^{(j)} & T_{2-1}^{(j)} & 0 & 0 \\ T_{-21}^{(j)} & T_{-2-1}^{(j)} & 0 & 0 \end{array}\right)\left(\begin{array}{c} \widehat{R}_{1,j}\\ \widehat{R}_{-1,j} \\ \widehat{R}_{2,j} \\  \widehat{R}_{-2,j} \end{array}\right)  \label{matrix}
\end{equation}
with
\[                                          
    \left(T_{il}^{(j)}\widehat{R}_{l,j}\right)(k,t) = \int\widehat{g}_{il}^{(j)}(k,k-m,\varepsilon  t) \widehat{R}_{l,j}(m,t) dm. 
\]
Then there exists a $q  > 0 $ such that if 
\begin{align}
\Vert\widehat{g}_{il}^{(j)}(\varepsilon t)\Vert_{X^{s,\varepsilon}}\le q \label{requirement}
\end{align}
holds for all $i,l\in\{\pm 1,\pm 2\}$, the transformation \eqref{Trafo} is bijective and has an 
inverse of the form 
\[
 \widehat{R}_{i,j}(k,t)=\widehat{R}_{i,j+1}(k,t)+\sum_{l\in\{1,-1,2,-2\}}\int\widehat{h}_{il}^{(j)}(k,k-m,\varepsilon t)\widehat{R}_{l,j+1}(m,t)dm 
\]
with 
\begin{align}
 \Vert \widehat{h}^{(j)}(\varepsilon t)\Vert_{X^{s,\varepsilon}}\le \frac{C \Vert\widehat{g}^{(j)}(\varepsilon t)\Vert_{X^{s,\varepsilon}}}{1-\Vert\widehat{g}^{(j)}(\varepsilon t)\Vert_{X^{s,\varepsilon}}} \label{h-norm},
\end{align}
where  $\Vert\widehat{g}^{(j)}(\varepsilon t)\Vert_{X^{s,\varepsilon}}=\max\limits_{i,l\in\{\pm 1,\pm 2\}}\{\Vert\widehat{g}_{il}^{(j)}(\varepsilon t)\Vert_{X^{s,\varepsilon}}\}$.
\end{lemma}

\begin{proof} Let $t$ be fixed. Obviously $T^{(j)}: \ \big(H^0_s\big)^4\rightarrow \big(H^0_s\big)^4$ is a linear operator.
Let $T_1^{(j)}$ and $T_2^{(j)}$ be defined as 
\[
 T_1^{(j)}=\left(\begin{array}{cc} T_{12}^{(j)} & T_{1-2}^{(j)} \\  T_{-12}^{(j)} & T_{-1-2}^{(j)}\end{array}\right) \qquad \textrm{and} \qquad T_2^{(j)}=\left(\begin{array}{cc} T_{21}^{(j)} & T_{2-1}^{(j)} \\  T_{-21}^{(j)} & T_{-2-1}^{(j)}\end{array}\right).
\]
We set $\Vert T^{(j)}\Vert:=\max \{\Vert T_1^{(j)}\Vert, \Vert T_2^{(j)}\Vert\}$ with 
\[
 \Vert T_1^{(j)}\Vert^2=\sup\limits_{\Vert (\widehat{R}_{2,j},\widehat{R}_{-2,j})^T\Vert\le1}\bigg(\sum_{l\in\{2,-2\}}\Vert T_{1l}^{(j)}(\widehat{R}_{l,j})\Vert_{H^0_s}^2+\sum_{l\in\{2,-2\}}\Vert T_{-1l}^{(j)}(\widehat{R}_{l,j})\Vert_{H^0_s}^2 \bigg) .
\]
By \eqref{garfunkel} we have 
$$
\Vert T_{il}^{(j)}\widehat{R}_{l,j}\Vert_{H^0_s} \le C \Vert\widehat{g}_{il}^{(j)}\Vert_{X^{s,\varepsilon}} \Vert\widehat{R}_{r,j}\Vert_{H^0_s}.
$$
Hence we have $\Vert T^{(j)}\Vert = \mathcal{O}(q) $. 
Therefore, for $ q > 0 $ sufficiently small we can use Neumann's series to invert
\begin{align}
 (I-(-T^{(j)}))^{-1}=\sum_{\lambda=0}^{\infty}\big(-T^{(j)}\big)^{\circ \lambda}\label{neumann}
\end{align}
where we denoted by $(T^{(j)}\big)^{\circ \lambda}$ the $\lambda$-times composition of $T^{(j)}$.
Thus we obtain a composition of operators as in Lemma \ref{lemma12}. For each pair $T_{il}^{(j)}$ and  $T_{st}^{(j)}$ we can write 
\begin{eqnarray*}
 (T_{il}^{(j)}\circ T_{st}^{(j)})\widehat{R}_{1,j+1}&=&\int \widehat{g}_{il}^{(j)}(k,k-m)\int \widehat{g}_{st}^{(j)}(m,m-n)\widehat{R}_{1,j+1}(n)dndm\\
             &=&\int {\int \widehat{g}_{il}^{(j)}(k,k-m)g_{st}^{(j)}(m,m-n)dm} \widehat{R}_{1,j+1}(n)dn.
\end{eqnarray*}
Hence we obtain inductively a series of integral kernels as in ~\eqref{inverse}.
The ${X^{s,\varepsilon}}$-norm of $\widehat{h}_{ik}^{(j)}$ is bounded by 
\begin{align*}
 \Vert h^{(j)}\Vert_{X^{s,\varepsilon}}\le C \sum_{l=1}^{\infty}\Big(\Vert g^{(j)}\Vert\Big)^l= \frac{C \Vert\widehat{g}^{(j)}\Vert_{X^{s,\varepsilon}}}{1-\Vert\widehat{g}^{(j)}\Vert_{X^{s,\varepsilon}}} .
\end{align*}
This is exactly \eqref{h-norm}.
\end{proof}

\subsection{Proof of  convergence}

In Lemma \ref{lemma12}, we assumed that \eqref{requirement} holds. Here we wish to show that all the $\hat{f}_{\cdot\,\cdot}^{(j)}$, $\hat{g}_{\cdot\,\cdot}^{(j)}$ and $\hat{h}_{\cdot\,\cdot}^{(j)}$ do in fact satisfy such estimates or even sharper estimates.
A simple application of Lemma \ref{lemma11} to \eqref{resonant} and \eqref{nonresonant} gives the estimates
\begin{align}
 \Vert \widehat{f}_{1,res}^{(j+1)}- \widehat{f}_{1,res}^{(j)}\Vert_{X^{s,\varepsilon}}\le& \ \sum_{\lambda\in\{2,-2\}}\Vert \widehat{\widetilde{g}}_{1\lambda}^{(j)} \Vert_{X^{s,\varepsilon}} \Vert \widehat{f}_{\lambda ,non}^{(j)} \Vert_{X^{s,\varepsilon}} +\Vert \widehat{f}_{1,res}^{(j)} \Vert_{X^{s,\varepsilon}} \sum_{\kappa\in\{1,-1\}}\Vert\widehat{h}_{\kappa \mu}^{(j)} \Vert_{X^{s,\varepsilon}} \nonumber\\
&+\sum_{\lambda\in\{2,-2\}}\Vert\widehat{\widetilde{g}}_{1\lambda}^{(j)} \Vert_{X^{s,\varepsilon}} \Vert\widehat{f}_{\lambda ,non}^{(j)} \Vert_{X^{s,\varepsilon}} \sum_{\kappa\in\{1,-1\}}\Vert \widehat{h}_{\kappa \mu}^{(j)}\Vert_{X^{s,\varepsilon}} \label{norm-r} \\
& +\sum_{\lambda\in\{2,-2\}}\Vert\widehat{\widetilde{g}}_{1\lambda}^{(j)} \Vert_{X^{s,\varepsilon}} \Vert \widehat{f}_{\lambda ,res}^{(j)}\Vert_{X^{s,\varepsilon}} \sum_{\kappa\in\{2,-2\}}\Vert\widehat{h}_{\kappa \mu}^{(j)} \Vert_{X^{s,\varepsilon}} \nonumber 
\end{align}
and
\begin{align}
\Vert\widehat{f}_{1,non}^{(j+1)} \Vert_{X^{s,\varepsilon}} \le& \sum_{\lambda\in\{2,-2\}}\Vert\widehat{\widetilde{g}}_{1\lambda}^{(j)} \Vert_{X^{s,\varepsilon}} \Vert\widehat{f}_{\lambda ,res}^{(j)} \Vert_{X^{s,\varepsilon}} +\Vert\widehat{f}_{1,res}^{(j)} \Vert_{X^{s,\varepsilon}} \sum_{\kappa\in\{1,-1\}}\Vert\widehat{h}_{\kappa\mu}^{(j)} \Vert_{X^{s,\varepsilon}} \nonumber\\
&+\sum_{\lambda\in\{2,-2\}}\Vert\widehat{\widetilde{g}}_{1\lambda}^{(j)} \Vert_{X^{s,\varepsilon}} \Vert\widehat{f}_{\lambda ,non}^{(j)} \Vert_{X^{s,\varepsilon}}\sum_{\kappa\in\{1,-1\}}\Vert\widehat{h}_{\kappa\mu}^{(j)} \Vert_{X^{s,\varepsilon}} \label{norm-n}  \\
&+\sum_{\lambda\in\{2,-2\}}\Vert\widehat{\widetilde{g}}_{1\lambda}^{(j)} \Vert_{X^{s,\varepsilon}} \Vert\widehat{f}_{\lambda ,res}^{(j)} \Vert_{X^{s,\varepsilon}}\sum_{\kappa\in\{2,-2\}}\Vert\widehat{h}_{\kappa\mu}^{(j)} \Vert_{X^{s,\varepsilon}}.\nonumber
\end{align}
One can obtain analogous inequalities for $\Vert\widehat{f}_{2,res}^{(j+1)}-\widehat{f}_{2,res}^{(j)}\Vert_{X^{s,\varepsilon}}$ and $\Vert\widehat{f}_{2,non}^{(j+1)}\Vert_{X^{s,\varepsilon}}$ with suitably adjusted indices. We will now give the important estimates mentioned earlier, from which our main convergence results will follow.

\begin{lemma}\label{theorem13}
There exists a $q > 0 $  such that for  
\begin{equation} \label{qeq}
\Vert\widehat{f}_{\nu ,res}^{(1)}\Vert_{X^{s,\varepsilon}}+\Vert\widehat{f}_{\nu ,non}^{(1)}\Vert_{X^{s,\varepsilon}} \leq q, \quad \nu\in\{1,-1,2,-2\}
\end{equation}
we have 
$$
\begin{array}{lll}
 {\bf a)} \ \Vert\widehat{f}_{\kappa ,res}^{(j)}\Vert_{X^{s,\varepsilon}}\le q\frac{1-q^{\frac{j}{2}}}{1-q^{\frac{1}{2}}},  \qquad &
 {\bf b)} \ \Vert\widehat{f}_{\kappa ,non}^{(j)}\Vert_{X^{s,\varepsilon}} \le q^{\frac{j+1}{2}}, \qquad &
{\bf c)} \ \Vert\widehat{g}_{\kappa\lambda}^{(j)}\Vert_{X^{s,\varepsilon}} \le C_{\omega}q^{\frac{j+1}{2}},  \\
{\bf d)} \ \Vert\widehat{h}_{\kappa\lambda}^{(j)}\Vert_{X^{s,\varepsilon}} \le 2C_{\omega}q^{\frac{j+1}{2}}, 
\qquad &
{\bf e)} \  \Vert \widehat{f}_{\kappa,res}^{(j+1)} - \widehat{f}_{\kappa,res}^{(j)} \Vert_{X^{s,\varepsilon}} \le 
 qq^{\frac{j}{2}} ,
\end{array}
$$
for all $j\in\mathbb{N}$
and $\kappa, \lambda\in\{2,-2,1,-1\}$, where 
$$C_{\omega}=\max\limits_{\mu\in\{1,-1\} \  \lambda\in\{2,-2\}}\sup\limits_{k\in\mathbb{R}}|i\omega_{\mu}(k)-i\omega_\lambda(k)|^{-1}.$$
\end{lemma}
 \begin{proof}
 The proof is based on an induction argument. \medskip
 
{\bf i)} For $j=1$ the estimates a) and b) follow from \eqref{qeq} and  the assertion in c) follows from  \eqref{2}. From \eqref{h-norm} we have 
\begin{align*}
\Vert h_{ik}^{(j)}\Vert_{X^{s,\varepsilon}}\le \frac{\Vert\widehat{g}^{(j)}\Vert_{X^{s,\varepsilon}}}{1-\Vert\widehat{g}^{(j)}\Vert_{X^{s,\varepsilon}}}. 
\end{align*}
If $q$ is chosen smaller than $\frac{1}{2C_{\omega}}$,  then due to the induction basis for c) we obtain the assertion in d) for $j=1$.\medskip

 {\bf ii)}  
 Using \eqref{2}, \eqref{h-norm} and \eqref{norm-r}  we obtain
 with the abbreviation $\Vert \widehat{f}_n^{(j)}\Vert_{X^{s,\varepsilon}}:=\max\limits_{\kappa\in\{1,-1,2,-2\}}\Vert \widehat{f}_{\kappa ,non}^{(j)}\Vert_{X^{s,\varepsilon}}$
 that
\begin{eqnarray*}
 \Vert \widehat{f}_{1,res}^{(j+1)}- \widehat{f}_{1,res}^{(j)} \Vert_{X^{s,\varepsilon}}\le&&2C_{\omega}\Vert \widehat{f}_{,non}^{(j)} \Vert_{X^{s,\varepsilon}}^2+4C_{\omega}\Vert \widehat{f}_{1,res}^{(j)} \Vert_{X^{s,\varepsilon}}\Vert\widehat{f}_{,non}^{(j)} \Vert_{X^{s,\varepsilon}} \\
&&+8C_{\omega}^2\Vert\widehat{f}_{,non}^{(j)} \Vert_{X^{s,\varepsilon}}^3+8C_{\omega}^2\Vert\widehat{f}_{,non}^{(j)} \Vert_{X^{s,\varepsilon}}^2\Vert \widehat{f}_{2,res}^{(j)}\Vert_{X^{s,\varepsilon}} \,.
\end{eqnarray*}
Using the induction hypotheses $\Vert\widehat{f}_{i,non}^{(j)}\Vert_{X^{s,\varepsilon}}\le q^{\frac{j+1}{2}}$ and $\Vert\widehat{f}_{i,res}^{(j)}\Vert_{X^{s,\varepsilon}}\le q\frac{1-q^{\frac{j}{2}}}{1-q^{\frac{1}{2}}}$ we find
\begin{eqnarray*}
 \Vert \widehat{f}_{1,res}^{(j+1)} - \widehat{f}_{1,res}^{(j)} \Vert_{X^{s,\varepsilon}} \le 
2C_{\omega}q^{j+1}+\frac{4C_{\omega}}{1-q^{\frac{1}{2}}}q^{\frac{j+3}{2}}+ 
8C_{\omega}^2q^{\frac{3(j+1)}{2}}+ \frac{8C_{\omega}^2}{1-q^{\frac{1}{2}}}q^{j+2}\le qq^{\frac{j}{2}} 
\end{eqnarray*}
which implies e) for $ q > 0 $ sufficiently small.
Since $ \widehat{f}_{1,res}^{(j+1)} = \widehat{f}_{1,res}^{(1)} + \sum_{k=1}^{j} (\widehat{f}_{1,res}^{(k+1)}-\widehat{f}_{1,res}^{(k)}) $ we  estimate
\begin{eqnarray*}
 \Vert \widehat{f}_{1,res}^{(j+1)}\Vert_{X^{s,\varepsilon}}\le q\sum_{k=0}^{j}q^{\frac{k}{2}}=q\frac{1-q^{\frac{j+1}{2}}}{1-q^{\frac{1}{2}}}.
\end{eqnarray*}
Using \eqref{norm-n} and \eqref{h-norm} we estimate  $ \Vert\widehat{f}_{i,non}^{(j)}\Vert_{X^{s,\varepsilon}}$  similarly to the resonant terms by
\begin{eqnarray*}
 \Vert\widehat{f}_{i,non}^{(j+1)}\Vert_{X^{s,\varepsilon}}&\le& 6C_\omega q^{\frac{j+1}{2}}q\frac{1-q^{\frac{j}{2}}}{1-q^{\frac{1}{2}}}
+8C_\omega q^{\frac{3(j+1)}{2}}+8C_\omega q^{j+2}\frac{1-q^{\frac{j}{2}}}{1-q^{\frac{1}{2}}}\le q^{\frac{j+2}{2}}
\end{eqnarray*}
for $ q > 0 $ sufficiently small.
 With the help of \eqref{h-norm} we obtain the estimates  \textbf{c)} and \textbf{d)}. 
\end{proof}
\begin{remark}{\rm 
 Due to  \eqref{eric} we have for 
  $\Vert \widehat{\Psi}\Vert_{L^1_s}$ sufficiently small  that $\Vert\widehat{f}_{i,res}^{(1)}\Vert_{X^{s,\varepsilon}}+\Vert\widehat{f}_{i,non}^{(1)}\Vert_{X^{s,\varepsilon}} \leq q $ 
  for a $ 0 < q \ll 1 $ independent of $ 0 < \varepsilon \ll 1 $.
 }
\end{remark}

\begin{remark}\label{remark15}
{\rm Similar estimates as in  Lemma \ref{theorem13}  hold for the partial derivatives of $\widehat{f}_{i,res}^{(j)}(k,l,\varepsilon t)$, $\widehat{g}_{ir}^{(j)}(k,l,\varepsilon t)$ and $\widehat{h}_{ir}^{(j)}(k,l,\varepsilon t)$ w.r.t. the first  and the  third argument, because the functions $\omega_{\lambda}(k)$ with $\lambda\in\{\pm 1,\pm 2\}$ are continuously differentiable and the  series for $\hat{h}_{\cdot\,\cdot}^{(j)}$ converges uniformly. The proof works in an absolute  analogous way. }
\end{remark}
Finally we prove 
\begin{lemma} \label{lemmashee}
The terms $\widehat{p}_{l,j+1}$ can be bound by 
\begin{eqnarray} \label{p-estimate1}
\|\widehat{p}_{l,j+1}\|_{H^0_s} &\leq& C \big(\, \| \widehat{{R}}_{-2,j+1} \|_{H^0_s} +\ldots +\| \widehat{{R}}_{2,j+1}\|_{H^0_s}  \\ && \qquad +  \varepsilon^{1/2} (\| \widehat{{R}}_{-2,j+1} \|_{H^0_s} +\ldots +\| \widehat{{R}}_{2,j+1} \|_{H^0_s})^2 + 1 \big) \nonumber
\end{eqnarray}
and 
\begin{equation}\label{stuck}
 \Vert \widehat{p}_{l,j+1}  - \widehat{p}_{l,j}\Vert_{H^0_s} \leq C q^{\frac{j}{2}},
\end{equation}
with  $ q \in (0,1) $ from Lemma \ref{theorem13} and with $ C $ a constant  independent of 
$ j $ and $ 0 < \varepsilon \ll 1 $.
\end{lemma}
\begin{proof}
Using  \eqref{epsilon} and Lemma \ref{lemma11} we find
\begin{eqnarray*}
 \Vert \widehat{p}_{1,j+1}  - \widehat{p}_{1,j}\Vert_{H^0_s} &\le & \sum_{\lambda\in\{2,-2\}}\Vert \partial_T \widehat{g}_{1,\lambda}^{(j)}\Vert_{X^{s,\varepsilon}} \Vert \widehat{R}_{\lambda,j}\Vert_{H^0_s}\\ && +\sum_{\lambda\in\{2,-2\}}\Vert \widehat{g}_{1,\lambda}^{(j)}\Vert_{X^{s,\varepsilon}}\Vert\widehat{p}_{\lambda,j}\Vert_{H^0_s}
+\Vert\widehat{r}_{1,j}\Vert_{H^0_s}.
\end{eqnarray*}
\cite[Lemma 3.1]{DS06} immediately implies 
$$ \| \widehat{r}_{1,j} \|_{H^0_s}  \leq  C \sum_{l\in\{2,-2\}} \|  \widehat{g}_{1l}^{(j)} \|_{X^{s,\varepsilon}} \| \widehat{R}_{l,j} \|_{H^0_s} ,$$
with $ C $ independent of $j $.
From Lemma \ref{theorem13} c) and Remark \ref{remark15} it follows
\begin{align}
 \Vert \widehat{p}_{,j+1}-\widehat{p}_{,j} \Vert_{H^0_s}\le & Cq^{\frac{j+1}{2}}\Vert\widehat{R}_{,j}\Vert_{H^0_s}+Cq^{\frac{j+1}{2}}\Vert\widehat{p}_{,j} \Vert_{H^0_s}, \nonumber
\end{align}
where
$\widehat{p}_{,j}= ( \widehat{p}_{-2,j},\widehat{p}_{-1,j},\widehat{p}_{1,j}, \widehat{p}_{2,j})$.
This inequality implies that $ (\widehat{p}_{,j})_{j \in \mathbb{N}} $ converges like a geometric series and that the $ \widehat{p}_{,j} $ can be estimated in terms 
of $ \widehat{p}_{,1} $.
\end{proof}

\section{The transformed equations and  the final energy estimates}

\label{sec4}
Due to Lemma \ref{theorem13} and Lemma \ref{lemmashee} we have for $l\in\{\pm 1, \pm 2\}$ that $\|\widehat{f}_{l,non}^{(j)}(\cdot,\varepsilon t)\|_{X^{s,\varepsilon}} \rightarrow 0$ for $j \rightarrow \infty$, the sequences $\big(\widehat{f}_{l,res}^{(j)}(\cdot,\cdot,\varepsilon t)\big)_{j\in\mathbb{N}}$ are Cauchy sequences in $X^{s,\varepsilon}$ and that the sequences $\big(\widehat{p}_{l,j}(\cdot,\varepsilon t)\big)_{j\in\mathbb{N}}$ and $\big(\widehat{R}_{l,j}(\cdot,\varepsilon t)\big)_{j\in\mathbb{N}}$ are Cauchy sequences in $H^{0}_s$. From the completeness  of ${X^{s,\varepsilon}}$ and $H^{0}_s$ it follows that their limits exist in $X^{s,\varepsilon}$ and $H^0_s$, respectively. Hence,
after the infinitely many transformations we can eliminate the non-resonant terms in \eqref{thunder1}-\eqref{thunder4} at leading order in $\varepsilon$. We  arrive at
\begin{eqnarray*}
\partial_t\widehat{\mathcal{R}}_1(k,t)&=&i\omega_1(k)\widehat{\mathcal{R}}_1(k,t)
\\ &&+i\omega_1(k) \int\sum_{\mu\in\{1,-1\}}\widehat{f}_{1,res}(k, k-m, \varepsilon t)\widehat{\mathcal{R}}_\mu(m,t)dm  +\varepsilon \widehat{P}_1(k,t),  \\
\partial_t\widehat{\mathcal{R}}_{-1}(k,t)&=&-i\omega_1(k)\widehat{\mathcal{R}}_{-1}(k,t)\\ &&-i\omega_1(k) \int\sum_{\mu\in\{1,-1\}}\widehat{f}_{1,res}(k, k-m, \varepsilon t)\widehat{\mathcal{R}}_\mu(m,t)dm+\varepsilon \widehat{P}_{-1}(k,t),  \\
\partial_t\widehat{\mathcal{R}}_2(k,t)&=&i\omega_2(k)\widehat{\mathcal{R}}_2(k,t)
\\ &&+i\omega_2^{-1}(k)\int\sum_{\mu\in\{2,-2\}}\widehat{f}_{2,res}(k, k-m, \varepsilon t)\widehat{\mathcal{R}}_\mu(m,t)dm+\varepsilon \widehat{P}_2(k,t), \\
\partial_t\widehat{\mathcal{R}}_{-2}(k,t)&=&-i\omega_2(k)\widehat{\mathcal{R}}_{-2}(k,t)\\ &&-i\omega_2^{-1}(k)\int\sum_{\mu\in\{2,-2\}}\widehat{f}_{2,res}(k, k-m, \varepsilon t)\widehat{\mathcal{R}}_\mu(m,t)dm+\varepsilon \widehat{P}_{-2}(k,t), 
\end{eqnarray*}
where $\widehat{\mathcal{R}}_{l} $, $ \widehat{f}_{l,res}$, $\widehat{P}_l$ are the limits of $\widehat{{R}}_{l,j} $, $\widehat{f}_{l ,res}^{(j)} $, $\widehat{p}_{l,j}$ for $j\rightarrow\infty$. 
The higher order terms can be estimated as 
\begin{eqnarray*} 
\|\widehat{P}_{l}\|_{H^0_s} &\leq& C \big(\, \| \widehat{\mathcal{R}}_{-2} \|_{H^0_s} +\ldots +\| \widehat{\mathcal{R}}_2 \|_{H^0_s}\\
&& \qquad +  \varepsilon^{1/2} (\| \widehat{\mathcal{R}}_{-2} \|_{H^0_s} +\ldots +\| \widehat{\mathcal{R}}_2 \|_{H^0_s})^2 + 1 \big) .
\end{eqnarray*}
Undoing the diagonalization \eqref{diagonali} yields
\begin{align}\label{undiagonal} 
\partial_t\widehat{\mathcal{R}}_u(k,t)=& i \omega_1(k)\widehat{\mathcal{W}}_u(k,t)+\varepsilon \widehat{P}_{u_1}(k,t),\\
\partial_t\widehat{\mathcal{W}}_u(k,t)=& i \omega_1(k)\widehat{\mathcal{R}}_u(k,t)+i \omega_1(k)\int \widehat{f}_u(k,k-m,\varepsilon  t) \widehat{\mathcal{R}}_u(m,t) dm + \varepsilon \widehat{P}_{u_2}(k,t),\nonumber \\
\partial_t\widehat{\mathcal{R}}_v(k,t)=& i \omega_2(k)\widehat{\mathcal{W}}_v(k,t)+\varepsilon \widehat{P}_{v_1}(k,t),\nonumber 
\\
\partial_t\widehat{\mathcal{W}}_v(k,t)=& i \omega_2(k)\widehat{\mathcal{R}}_v(k,t)+i \omega_2^{-1}(k) \int \widehat{f}_v(k,k-m,\varepsilon  t) \widehat{\mathcal{R}}_v(m,t) dm +  \varepsilon \widehat{P}_{v_2}(k,t), \nonumber
\end{align}
with the abbreviation
$$
\widehat{f}_u(k,k-m,\varepsilon t) =2\widehat{f}_{1,res}(k,k-m,\varepsilon t), \qquad
\widehat{f}_v(k,k-m,\varepsilon t) =2\widehat{f}_{2,res}(k,k-m,\varepsilon t)
$$
and 
\begin{eqnarray*} 
\|\widehat{P}_{w}\|_{H^0_s} &\leq& C \Big(\, \| \widehat{\mathcal{R}}_{u} \|_{H^0_s} +\ldots +\| \widehat{\mathcal{W}}_v \|_{H^0_s} \\ && \qquad +  \varepsilon^{1/2} (\| \widehat{\mathcal{R}}_{u} \|_{H^0_s} +\ldots +\| \widehat{\mathcal{W}}_v \|_{H^0_s})^2 + 1 \big) 
\end{eqnarray*}
for $w \in \{u_1,u_2,v_1,v_2\}$ and $C$ only depending on $\omega_1, \omega_2$ and $\widehat{\Psi}$.

The energy estimates for \eqref{bousseq1} are based on integration by parts.
The following lemma allows to transfer this approach to \eqref{undiagonal}. 
\begin{lemma}\label{lemma property}
For $w\in\{u,v\}$
 the functions $\widehat{f}_{w}(k,k-m,\varepsilon t)$  satisfy:
\begin{enumerate}
 \item[(i)]  $\widehat{f}_{w}(k,k-m,\varepsilon t) =\overline{\widehat{f}_{w}(k,m-k,\varepsilon t) }$,
  \item[(ii)] $\sup\limits_{k \in \R} \Vert\widehat{f}_{w}(k,\cdot,\varepsilon t)-\widehat{f}_{w}(k -\cdot,\cdot,\varepsilon t)\Vert_{L^1} \leq C\varepsilon.$
\end{enumerate}
\end{lemma}
\begin{proof}
 The starting point of the iteration process was 
 $$ \widehat{f}_{\lambda ,res}^{(1)}(k,k-m,\varepsilon t)= \widehat{\Psi}(k-m,\varepsilon t), $$ 
 where  $\widehat{\Psi}(\cdot,\varepsilon t)$ is the Fourier transform of a real-valued function, such that 
$\widehat{\Psi}(k,\cdot)=\overline{\widehat{\Psi}(-k,\cdot)}$. Therefore, we obtain 
\begin{align}
\widehat{f}_{\lambda ,res}^{(j)}(k,k-m,\varepsilon t)=&\overline{\widehat{f}_{\lambda ,res}^{(j)}(k,m-k,\varepsilon t)},\label{eigenschaft} \\
\widehat{f}_{\lambda ,non}^{(j)}(k,k-m,\varepsilon t)= &\overline{\widehat{f}_{\lambda ,non}^{(j)}(k,m-k,\varepsilon t)}  \nonumber 
\end{align}
for $j=1$. Using \eqref{2} we find that $\widehat{g}_{\lambda \mu}^{(1)}(k,k-m,\varepsilon t)=\overline{\widehat{g}_{\lambda \mu}^{(1)}(k,m-k,\varepsilon t)}$
and hence 
\[
 \widehat{h}_{\lambda \mu}^{(1)}(k,k-m,\varepsilon t)=\overline{\widehat{h}_{\lambda \mu}^{(1)}(k,m-k,\varepsilon t)}.
\]
 With the help of \eqref{resonant} and a simple induction, it follows that \eqref{eigenschaft} holds for all $j\in\mathbb{N}$. Hence, the assertion in {\it (i)} is valid. 
 
Property (ii) holds due the concentration of  $ \widehat{\psi}_n $ at the wave number $k=0$. In detail, we have
\begin{eqnarray*}
 &&\sup_{k \in \R} \Vert\widehat{f}_{w}(k,\cdot,\varepsilon t)-\widehat{f}_{w}(k-\cdot,\cdot,\varepsilon t)\Vert_{L^1}\\ &\le&\int\sup\limits_{\xi \in \R}|\partial_1\widehat{f}_w(\xi,l)|(1+(l/\varepsilon)^2)^{s/2}\frac{l}{(1+(l/\varepsilon)^2)^{s/2}}dl\\ &\leq& C \varepsilon \Vert\partial_1\widehat{f}_w\Vert_{X^{s,\varepsilon}} .
\end{eqnarray*}
due to Lemma \ref{theorem13} and Remark \ref{remark15}. This implies (ii). 
\end{proof}

In order to establish Theorem \ref{theorem1}
we will prove an $\mathcal{O}(1)$-bound for an  energy $ \mathcal{E}_s$ which will be constructed in the following. It will be close to $ E_s  = E_{s,u}+ E_{s,v}$  which is defined by
\begin{align}
E_{s,u}(t) = \sum\limits_{j=0}^{s} \int |k|^{2j} \big(|\widehat{\mathcal{R}}_u|^2(k,t) + |\widehat{\mathcal{W}}_u|^2(k,t) \big) dk \label{energy_u},\\
E_{s,v}(t) = \sum\limits_{j=0}^{s-1} \int |k|^{2j}\omega_2^{2}(k) \big(|\widehat{\mathcal{R}}_v|^2(k,t) + \widehat{|\mathcal{W}}_v|^2(k,t) \big) dk \label{energy_v}.
\end{align}

Computing $\displaystyle \frac{d}{dt}E_{s,u}(t)$ and realizing that the autonomous linear terms from \eqref{undiagonal} cancel gives 
\begin{eqnarray*}
\frac{d}{dt}E_{s,u}(t) & = &  \sum\limits_{j=0}^{s} \int \int |k|^{2j} i \omega_1(k)\overline{\widehat{\mathcal{W}}_u}(k,t) \widehat{f}_u(k,k-m,\varepsilon  t) \widehat{\mathcal{R}}_u(m,t) dm dk \\&& 
 - \sum\limits_{j=0}^{s} \int \int |k|^{2j} i \omega_1(k)\widehat{\mathcal{W}}_u(k,t)  \overline{\widehat{f}_u(k,k-m,\varepsilon  t) \widehat{\mathcal{R}}_u(m,t)} dm dk +\varepsilon \widehat{Q}_1(t) 
\\& = & - \sum\limits_{j=0}^{s} \int \int |k|^{2j} \partial_t \overline{\widehat{\mathcal{R}}_u}(k,t) \widehat{f}_u(k,k-m,\varepsilon  t) \widehat{\mathcal{R}}_u(m,t) dm dk \\&& 
 - \sum\limits_{j=0}^{s} \int \int |k|^{2j} \partial_t \widehat{\mathcal{R}}_u(k,t)  \overline{\widehat{f}_u(k,k-m,\varepsilon  t) \widehat{\mathcal{R}}_u(m,t)} dm dk  +\varepsilon \widehat{Q}_2(t) 
 \end{eqnarray*}
 where we used the first line of \eqref{undiagonal}. 
Moreover, we have 
\begin{eqnarray*}
|\widehat{Q}_l| & \leq & C (E_s +  \varepsilon^{1/2} (E_s)^{3/2}+1)
\end{eqnarray*}
for $l=1,2$.
 Interchanging the role of $ k $ and $ m $ in the 
 second term yields 
 \begin{eqnarray*}
\frac{d}{dt}E_{s,u}(t)  & = & - \sum\limits_{j=0}^{s} \int \int |k|^{2j} \partial_t \overline{\widehat{\mathcal{R}}_u}(k,t) \widehat{f}_u(k,k-m,\varepsilon  t) \widehat{\mathcal{R}}_u(m,t) dm dk \\&& 
 - \sum\limits_{j=0}^{s} \int \int |m|^{2j} \partial_t \widehat{\mathcal{R}}_u(m,t)  \overline{\widehat{f}_u(m,m-k,\varepsilon  t) \widehat{\mathcal{R}}_u(k,t)} dm dk+\varepsilon \widehat{Q}_2(t) .
 \end{eqnarray*}
 Using that $ k^j = m^j + \mathcal{O}(k-m) $ and Lemma \ref{lemma property} gives next 
  \begin{eqnarray*}
\frac{d}{dt}E_{s,u}(t)   & = &-  \sum\limits_{j=0}^{s} \int \int |k|^{j} \partial_t \overline{\widehat{\mathcal{R}}_u}(k,t) \widehat{f}_u(k,k-m,\varepsilon  t) |m|^{j} \widehat{\mathcal{R}}_u(m,t) dm dk \\&& 
 - \sum\limits_{j=0}^{s} \int \int |m|^{j} \partial_t \widehat{\mathcal{R}}_u(m,t)\widehat{f}_u(m,k-m,\varepsilon  t) |k|^{j}  \overline{\widehat{\mathcal{R}}_u(k,t)} dm dk +\varepsilon \widehat{Q}_3(t) 
\\Š& = & - \sum\limits_{j=0}^{s} \int \int |k|^{j} \partial_t \overline{\widehat{\mathcal{R}}_u}(k,t) \widehat{f}_u(k,k-m,\varepsilon  t) |m|^{j} \widehat{\mathcal{R}}_u(m,t) dm dk \\&& 
 - \sum\limits_{j=0}^{s} \int \int |m|^{j} \partial_t \widehat{\mathcal{R}}_u(m,t)\widehat{f}_u(k,k-m,\varepsilon  t)  |k|^{j} \overline{\widehat{\mathcal{R}}_u(k,t)} dm dk \\ && + \varepsilon Q_4(t) +\varepsilon \widehat{Q}_3(t) ,
\end{eqnarray*}
with 
$$ \varepsilon Q_4(t)=-\sum\limits_{j=0}^{s} \int\int|m|^{j}\partial_t\widehat{\mathcal{R}}_u(m,t)\big(\widehat{f}_u(k,k-m,\varepsilon  t)-\widehat{f}_u(m,k-m,\varepsilon  t)\big)|k|^{j}\widehat{\mathcal{R}}_u(k,t)dmdk.$$
H\"older's inequality, Young's inequality, Lemma \ref{lemma property} and \eqref{undiagonal} allows us to prove for $ Q_3 $ and $ Q_4 $ the same estimates as 
for $ Q_1 $ and $ Q_2 $.
Therefore we obtain
\begin{eqnarray*}
\frac{d}{dt} E_{s,u}(t)& = & - \sum\limits_{j=0}^{s} \int \int \partial_t( |k|^{j} \overline{\widehat{\mathcal{R}}_u}(k,t) |m|^{j} \widehat{\mathcal{R}}_u(m,t)) \widehat{f}_u(k,k-m,\varepsilon  t) dm dk +\varepsilon \widehat{Q}_5(t)
\\& = & - \sum\limits_{j=0}^{s} \partial_t\left(\int \int |k|^{j} \overline{\widehat{\mathcal{R}}_u}(k,t) |m|^{j} \widehat{\mathcal{R}}_u(m,t) \widehat{f}_u(k,k-m,\varepsilon  t) dm dk\right) +\varepsilon \widehat{Q}_6(t)
\end{eqnarray*}
with
\begin{eqnarray*}
|\widehat{Q}_l| &\leq& C (E_s +  \varepsilon^{1/2} (E_s)^{3/2} + 1)
\end{eqnarray*}
for $l=5,6$.
Doing the same calculations  for
$E_{s,v}$ finally yields 
\begin{eqnarray*}
\partial_t \mathcal{E}_s &\leq& C \varepsilon\, (\mathcal{E}_s +  \varepsilon^{1/2} (\mathcal{E}_s)^{3/2}+1)   ,
\end{eqnarray*}
where
\begin{eqnarray*}
\mathcal{E}_s(t)&=&E_s(t)+ \sum\limits_{j=0}^{s} \int \int |k|^{j} \overline{\widehat{\mathcal{R}}_u}(k,t) |m|^{j}\widehat{\mathcal{R}}_u(m,t) \widehat{f}_u(k,k-m,\varepsilon  t) dm dk\\ &&+ \sum\limits_{j=0}^{s} \int \int |k|^{j} \overline{\widehat{\mathcal{R}}_v}(k,t) |m|^{j}\widehat{\mathcal{R}}_v(m,t) \widehat{f}_v(k,k-m,\varepsilon  t) dm dk.
\end{eqnarray*}
Now, a simple application of Gronwall's inequality yields an $\mathcal{O}(1)$-bound for $\mathcal{E}_s$ for all $t\in[0,T_0/\varepsilon]$ for $\varepsilon>0$ sufficiently small. Since $\|R_u\|_{H^{s}}+\|R_v\|_{H^{s}}+\|W_u\|_{H^{s}}+\|W_v\|_{H^{s}} \leq \sqrt{\mathcal{E}_s}$ for $\Psi$ sufficiently small we obtain
\[
\sup\limits_{t\in[0,T_0/\varepsilon]}\|(u,v)(\cdot,t)-(U,V)(\varepsilon\,\cdot,\varepsilon t)\|_{H^s}\;=\;\varepsilon^{3/2}\sup\limits_{t\in[0,T_0/\varepsilon]}\| R(\cdot,t) \|_{H^s}\;\le\; C\varepsilon^{3/2}.
\]
With the help of  Sobolev's embedding theorem we finally get the  bound stated in Theorem \ref{theorem1}. 
\hfill $\square$


\section{Discussion}
\label{secdiscussion}

The question occurs whether to a given initial condition 
the associated solution $(u,v)$ of the KGB system can be approximated by the Whitham approximation.
\begin{theorem} \label{KAnord}
Consider the KGB system \eqref{eq1}-\eqref{eq2} with the family of  initial conditions 
$$
(u,\partial_t u,v,\partial_t v)(x,0) = (\Phi_1,\varepsilon  \Phi_2,H(\Phi_1),\varepsilon H'(\Phi_1) \Phi_2)
(\varepsilon  x)
$$ 
parametrized by $ \varepsilon   > 0 $. Then for every $ T_0 > 0 $  
there exist $ C > 0 $ and $ \varepsilon_0   > 0 $ such that 
for  $ \Phi_j $ satisfying
$ \| \Phi_1 \|_{H^5} + \| \partial_X^{-1} \Phi_2  \|_{H^5}  \leq C $
and every $ \varepsilon \in (0,\varepsilon_0) $
there are solutions $ U \in C([0,T_0],H^5) $ of the Whitham equation  
\eqref{eq5}  with initial conditions $ U(X,0) = \Phi_1(X) $ and $ \partial_T U(X,0) = \Phi_2(X) $
such that 
\[
\sup\limits_{t\in [0,T_0/ \varepsilon]} \|(u,v)(\cdot,t)-(U,H(U))(\varepsilon \,\cdot,\varepsilon t)\|_{H^{1}} \le C_2\varepsilon^{3/2}.
\]
\end{theorem}
\begin{remark}{\rm  \label{KAsued}
The previous assumptions on $v$ can be weakened slightly  to 
$$
\| (v,\varepsilon^{-1} \partial_t v)(\cdot,0) - (H(\Phi_1),H'(\Phi_1) \Phi_2)
(\varepsilon  \cdot) \|_{H^5\times H^4}
\leq C \varepsilon^{3/2}.
$$
}
\end{remark} 
The proof of Theorem \ref{KAnord} and of Remark \ref{KAsued} is a direct consequence 
of the  estimates shown in Section \ref{sec4}.
The local existence and uniqueness of solutions for the Whitham equation  
\eqref{eq5} resp. for \eqref{eq5a}-\eqref{eq5b} follows from \cite{Kato75}.
It is not expected that the set of initial conditions for the KGB system 
for which the associated solutions can be described by the 
Whitham approximation is much larger.   


\end{document}